\documentclass[a4paper,12pt]{article}

\usepackage{amsthm}
\usepackage{mathrsfs}
\usepackage{palatino}
\usepackage[all]{xy, xypic}
\usepackage{amsfonts,amsmath,amssymb,amscd}

\usepackage{epsfig}

\usepackage{euscript}
\usepackage{amscd}
\usepackage{amsthm}
\usepackage{theoremref}
\usepackage{tikz-cd}
\usepackage[all]{xy}
\DeclareMathAlphabet{\mathpzc}{OT1}{pzc}{m}{it}

\newtheorem{thm}{Theorem}[section]

\newtheorem{lem}[thm]{Lemma}
\newtheorem{prop}[thm]{Proposition} 
\newtheorem{cor}[thm]{Corollary}
\newtheorem{rem}[thm]{Remark}
\newtheorem{ex}[thm]{Example}

\newcommand{\p}{\mathpzc{p}}
\newcommand{\q}{\mathpzc{q}}

\newcommand{\bC}{\mathbb C}

\newcommand{\A}{\mathbb A}

\newcommand{\Spec}{\operatorname{Spec}}

\newcommand{\td}{\operatorname{tr.deg}}
\newcommand{\End}{\operatorname{End}}
\newcommand{\Aut}{\operatorname{Aut}}

\newcommand{\dk}{\operatorname{DK}}
\newcommand{\ml}{\operatorname{ML}}

\title{On the triviality of a family of linear hyperplanes}
\author{Parnashree Ghosh$^*$ and Neena Gupta$^{**}$\\
	{\small{\it  Theoretical Statistics and Mathematics  Unit, Indian Statistical Institute,}}\\ 
	{\small{\it 203 B.T.Road, Kolkata-700108, India}}\\
	{\small{\it e-mail : $^*$parnashree$\_$r@isical.ac.in, ghoshparnashree@gmail.com}}\\
	{\small{\it e-mail : $^{**}$neenag@isical.ac.in, rnanina@gmail.com}}
}

\begin{document}
	
	\date{}
	\maketitle
	\abstract{Let $k$ be a field, $m$ a positive integer, $\mathbb{V}$ an affine subvariety of $\mathbb{A}^{m+3}$
		defined by a linear relation of the form $x_{1}^{r_{1}}\cdots x_{m}^{r_{m}}y=F(x_{1}, \ldots , x_{m},z,t)$,
		$A$ the coordinate ring of $\mathbb{V}$ and $G= X_1^{r_1}\cdots X_m^{r_m}Y-F(X_1, \dots, X_m,Z,T)$. In \cite{com}, 
		the second author had studied the case $m=1$ and had obtained several necessary 
		and sufficient conditions for $\mathbb{V}$ to be isomorphic to the affine 3-space and $G$ to be a coordinate
in $k[X_1, Y,Z,T]$. 

In this paper, we study the general higher-dimensional
		variety $\mathbb{V}$ for each $m \geqslant 1$ and
		obtain analogous conditions for $\mathbb{V}$ to be isomorphic to $\mathbb{A}^{m+2}$ and $G$ to be a coordinate
in $k[X_1, \dots, X_m, Y,Z,T]$, under a certain hypothesis on $F$. Our main theorem immediately yields a family of 
		higher-dimensional linear hyperplanes for which the Abhyankar-Sathaye Conjecture holds. 
	
		We also describe the isomorphism classes and automorphisms of integral domains of the type $A$ under certain conditions. 
		These results show that for each $d \geqslant 3$, there is a family of infinitely many pairwise
		non-isomorphic rings which are counterexamples to the Zariski Cancellation Problem 
		for dimension $d$ in positive characteristic.

	\smallskip
	
		

	\section{Introduction}

Throughout this paper, $k$ will denote a field and for a commutative ring $R$, 
$R^{[n]}$ will denote a polynomial ring in $n$ variables over $R$. 
$R^{*}$ will denote the group of units in $R$.

We begin with the Epimorphism Problem, one of the fundamental problems in the area of Affine Algebraic Geometry (cf. \cite{DG}, \cite{icm}):

	\medskip
	\noindent
	\textbf{Question 1.} If $\frac{k[X_{1},\ldots,X_{n}]}{(H)}=k^{[n-1]}$, then is $k[X_{1},\ldots,X_{n}]=k[H]^{[n-1]}$?	 
    
	\medskip

The famous Epimorphism Theorem of S.S. Abhyankar and T. Moh (\cite{AM}), also proved independently by M. Suzuki (\cite{Suz}), asserts an affirmative answer to Question 1 when $k$ is a field of characteristic zero and $n=2$. The Abhyankar-Sathaye Conjecture asserts an affirmative answer to Question 1 when $k$ is of
characteristic zero and $n>2$; and this remains a formidable open problem in Affine Algebraic Geometry.

Recall that, when $k$ is a field of positive characteristic,  explicit counterexamples to Question 1 had already been demonstrated by B. Segre (\cite{Se}) in 1957 and  M. Nagata (\cite{Na}) in 1971. However, when the hyperplane $H$ is of some specified form,  
it is possible to obtain affirmative answers to Question 1 even when $k$ is of arbitrary characteristic. Thus, the Abhyankar-Sathaye Conjecture
could be extended to fields of arbitrary characteristic for certain special cases of $H$.

The first (partial) affirmative solution to Question 1 was obtained for the case $n=3$ and $H$ a linear plane, 
i.e., linear in one of the three variables, by A. Sathaye (\cite{sp}) in characteristic zero and  P. Russell (\cite{rp}) in arbitrary characteristic.
They also proved that if $A=k^{[2]}$ and the hyperplane $H \in A[Y](=k^{[3]})$ is of the form $aY+b$, where $a,b \in A$, then the coordinates $X, Z$ of
$A$ can be chosen such that $A=k[X,Z]$ with $a \in k[X]$; that is,  linear planes were shown to be of the form $a(X)Y+b(X,Z)$.    

For the case $n=4$ and $k=\mathbb C$,  S. Kaliman et.al. (\cite{kvz}) proved the Abhyankar-Sathaye Conjecture  for certain linear hyperplanes 
in ${\mathbb C} [X_1, X_2, Y, Z]$ of the type $a(X_1)Y+b(X_1,X_2, Z)$ and $a(X_1, X_2)Y+b(X_1,X_2, Z)$ under certain hypotheses. 
A general survey on other partial affirmative answers to the Abhyankar-Sathaye Conjecture has been made in \cite[Section 2]{DG}. 

In this paper we consider certain types of linear hyperplanes in higher dimensions which arose out of investigations on the Zariski Cancellation Problem 
in \cite{adv}. We present the genesis below.

Consider a ring of the form
	\begin{equation}\label{A}
		A:= \frac{k\left[ X_{1}, \ldots , X_{m}, Y,Z,T\right]}
{\left(X_{1}^{r_{1}} \cdots X_{m}^{r_{m}}Y -F(X_{1}, \ldots , X_{m}, Z,T) \right)}, 
		~~r_{i} >1  \text{~for all~}i, 1 \leqslant i \leqslant m,
	\end{equation}
		where $F(0,\ldots,0,Z,T) \neq 0$. Set $f(Z,T):=F(0, \ldots, 0, Z,T)$. Let $x_1,\ldots,x_m,y,z,t$ denote the images in $A$ of $X_1,\ldots,X_m,Y,Z,T$ respectively.

Now suppose $F=f(Z,T)$, where $f$ is a {\it line} in $k[Z,T]$, i.e.,  $\frac{k[Z,T]}{\left( f \right)} = k^{[1]}$. 
	For each integer $m \geqslant 1$, the second author established the following two properties of the integral domain $A$ under the above hypothesis \cite[Theorem 3.7]{adv}:
	 
	 \smallskip
	 \noindent
	  (a) $A^{[1]}=k^{[m+3]}$.

	 \smallskip
	\noindent
	  (b) If $A=k^{[m+2]}$,
	   then $k[Z,T]=k[f]^{[1]}$. 
	 
	 \smallskip
	 
Now assume further that $ch \,k:=p >0$. Let $f(Z,T)$ be a Segre-Nagata 
{\it non-trivial line} in $k[Z,T]$ (i.e., ${k[Z,T]} \neq k[f]^{[1]}$), 
for instance, consider $f(Z,T)=Z^{p^{e}}+T+T^{sp}$, where
	 $p^{e} \nmid sp$ and $sp \nmid p^{e}$. From (a) and (b) it 
follows that the corresponding ring $A$ is stably isomorphic to $k^{[m+2]}$ but not isomorphic to $k^{[m+2]}$.
Therefore, for each $m \geqslant 1$, such a ring $A$ is a counterexample to the Zariski Cancellation
 Problem in positive characteristic in each dimension greater than or equal to three \cite[Corollary 3.8]{adv}.

	  Earlier, the second author had investigated the ring $A$ for $m=1$ and had shown that the 
condition (b) holds (i.e., $A= k^{[3]}$ implies $k[Z,T]=k[f]^{[1]}$)
even without the hypothesis that $f(Z,T)$ is a line in $k[Z,T]$. In fact, she had proved the following general result \cite[Theorem 3.11]{com}:
	
	\medskip
	
	\noindent
	{\bf Theorem A.}
	{\it	Let $k$ be a field and $A^{\prime}=\frac{k[X,Y,Z,T]}{(X^{r}Y-F(X,Z,T))}$, where $r>1$. Let $x$ be the image of $X$ in $A^{\prime}$. If $G:= X^{r}Y-F(X,Z,T)$ and $f(Z,T) :=F(0,Z,T) \neq 0$, then the following statements are equivalent: }
		\begin{enumerate}
			\item [(i)] $k[X,Y,Z,T]=k[X,G]^{[2]}$.
			
			\item [(ii)] $k[X,Y,Z,T]=k[G]^{[3]}$.
			
			\item [(iii)] $A^{\prime}=k[x]^{[2]}$.
			
			\item [(iv)] $A^{\prime}=k^{[3]}$.
			
			\item [(v)] $k[Z,T]=k[f]^{[1]}$.

		\end{enumerate}
	
	Theorem A provides a common framework for understanding the non-triviality of the Russell-Koras threefold 
(over $\mathbb{C}$) and the Asanuma threefold (in positive characteristic), both of which have been central objects
 of study in Affine Algebraic Geometry (cf. \cite[Section 4]{Ijp}). Also note that the equivalence of (ii) and (iv)
 above establishes a special case of the Abhyankar-Sathaye Conjecture for linear hyperplanes in $k^{[4]}$.

    \smallskip
	In view of the importance of Theorem A, A. K. Dutta had asked the authors whether similar results as Theorem A
 hold over the ring $A$ (as in \eqref{A}) when $m >1$, in particular: 
	
	\smallskip
	\noindent
	{\bf Question 2.} For $m>1$, is the condition $k[Z,T]=k[f]^{[1]}$ both necessary and sufficient for the ring $A$ (as in \eqref{A}) to be $k^{[m+2]}$ ?
     
     \smallskip	
	
	 \thref{ex1} shows that when $m \geqslant 2$, the condition that $f$ is a coordinate in $k[Z,T]$ is not sufficient for $A$ to be $k^{[m+2]}$ in general. However in section 3, we will show that Question 2 indeed has an affirmative
	 answer when $F$ is of the form
 	$$
	F(X_{1}, \ldots, X_{m}, Z,T)= f(Z,T)+ (X_{1} \cdots X_{m})g(X_{1}, \ldots , X_{m},Z,T).
	$$
	 In fact, we prove the following generalisation of Theorem A (\thref{equiv}):
	
 \medskip
 
 \noindent
 {\bf Theorem B.}
 	{\it	Let $k$ be a field and 
 		$$
 		A= \frac{k\left[ X_{1}, \ldots , X_{m}, Y,Z,T\right]}{\left(X_{1}^{r_{1}} \cdots X_{m}^{r_{m}}Y- F(X_{1}, \ldots, X_{m}, Z,T) \right)},
		~~r_{i} >1  \text{~for all~} i, 1 \leqslant i \leqslant m.
 		$$  
		Let  $F(X_{1}, \ldots , X_{m},Z,T)=f(Z,T)+ (X_{1} \cdots X_{m})g(X_{1}, \ldots , X_{m},Z,T)$ 
		be such that $f(Z,T) \neq 0$. Let $G=X_{1}^{r_{1}} \cdots X_{m}^{r_{m}}Y- F(X_{1}, \ldots, X_{m}, Z,T)$ 
		and $x_{1}, \ldots , x_{m}$ denote the images in $A$ of $X_{1}, \ldots , X_{m}$ respectively.  Then the following statements are equivalent:}
 	\begin{enumerate}
 		\item [(i)] $k[X_{1}, \ldots , X_{m},Y,Z,T]=k[X_{1}, \ldots, X_{m},G]^{[2]}$.
 		
 		\item[(ii)]  $k[X_{1}, \ldots , X_{m},Y,Z,T]=k[G]^{[m+2]}$.
 		
 		\item[(iii)] $A=k[x_{1}, \ldots , x_{m}]^{[2]}$.
 		
 		\item[(iv)] $A=k^{[m+2]}$.
 		
 		\item[(v)] $k[Z,T]=k[f(Z,T)]^{[1]}$.

 	\end{enumerate}
We actually prove (Theorems \ref{equiv} and \ref{eq13}) the 
equivalence of each of the above statements with eight other technical statements, involving stable isomorphisms, affine fibrations
 and two invariants --- the {\it Derksen invariant} and the {\it Makar-Limanov invariant}
(both the invariants are defined in section 2).  The main theorem (Theorem \ref{equiv}) provides a
 connection between the Epimorphism Problem, Zariski Cancellation Problem and Affine Fibration Problem in higher dimensions.
We also establish a generalisation of Theorem B over any commutative Noetherian domain $R$ such that either $R$ is seminormal or $\mathbb{Q} \subset R$ (\thref{geqv}).

The equivalence of (v) with any of the remaining statements in Theorem B shows that 
a certain property of the polynomial $G$ in $m+3$ variables (the property of being a coordinate)
is determined entirely by a property of the polynomial $f$ in two variables, i.e., a question on 
a polynomial in $m+3$ variables reduces to a question on a polynomial in 2 variables.
The equivalence of (iv) and (v) answers Question 2 for the particular structure of $F$; 
the equivalence of (ii) and (iv) gives an affirmative answer to Question 1 
 (Abhyankar-Sathaye Conjecture) for $n \geqslant 4$
 and for a hypersurface of the form $G$. This result may be considered as a partial generalisation of the theorem on Linear Planes due to A. Sathaye (\cite{sp}) and P. Russell (\cite[2.3]{rp}), mentioned earlier. 
 
In section 4, we consider rings of type $A$ (as in (\ref{A}))
when the Derksen invariant  $\dk(A)=k[x_1,\ldots,x_m,z,t]$. 
We first describe some necessary conditions for two such rings to be isomophic (\thref{isocl}). 
Next we deduce some necessary and sufficient conditions for an endomorphism of any such $A$ to be an automorphism (\thref{aut1}). 
Finally, we determine the isomorphism classes of the family of rings defined by $A$, 
when $F=f(Z,T)$ and $f(Z,T)$ is a non-trivial line 
(\thref{isocl1}). The results yield an infinite family 
of pairwise non-isomorphic affine varieties of higher dimensions ($\geqslant 3$),
which are counterexamples to the Zariski Cancellation Problem in positive characteristic (\thref{niso}).

 In the next section, we recall the notion of exponential maps, Grothendieck groups and some earlier results which will be used in this paper.

\section{Preliminaries}
	
All rings and algebras will be assumed to be commutative with unity. We first fix some notation.
Capital letters like $X,Y,Z,T, U,V, X_1, \dots, X_m$, etc, will denote indeterminates
over the respective ground rings or fields. 
	
     Let $R$ be a ring and $B$ an $R$-algebra. For a prime ideal $\p$ of $R$, 
let $k(\p)$ denote the field $\frac{R_{\p}}{\p R_{\p}}$ and $B_{\p}$ 
denote the ring $S^{-1}B$, where $S=R \setminus \p$. Thus 
     $\frac{B_{\p}}{\p B_{\p}}=B\otimes_R k(\p)$. 
	 Recall that $k$ will always denote a field and $R^{[n]}$ a polynomial ring in $n$ variables over $R$.

	 We first recall a few definitions. 
	 
	 \medskip
	 \noindent 
	 {\bf Definition.} A polynomial $h \in k[X,Y]$ is said to be a {\it line} in $k[X,Y]$ if $\frac{k[X,Y]}{(h)} = k^{[1]}$. Furthermore, if $k[X,Y] \neq k[h]^{[1]}$, then $h$ is said to be a {\it non-trivial line} in $k[X,Y]$. 
	 
   \medskip
   \noindent
   {\bf Definition.} A finitely generated flat $R$-algebra $B$ is said to be an {\it $\A^n$-fibration} over $R$
   	if $B \otimes_R k(\p) = k(\p)^{[n]}$ for every prime ideal $\p$ of $R$.

   \medskip
   \noindent
   {\bf Definition.}  A polynomial $h \in R[X,Y]$ is said to be a {\it residual coordinate}, if for every
  $\p \in \Spec(R)$, $$
R[X,Y] \otimes_{R} \kappa(\p)=(R[h] \otimes_{R} \kappa(\p) )^{[1]}.
$$

 We now define an exponential map on a $k$-algebra $B$ and two invariants related to it, namely, the {\it Makar-Limanov invariant}  and the {\it Derksen invariant}.
 
  \medskip
 \noindent
 {\bf Definition.}
 	Let $B$ be a $k$-algebra and $\phi: B \rightarrow B^{[1]}$ be a $k$-algebra homomorphism. For an indeterminate $U$ over $B$, 
let $\phi_{U}$ denote the map $\phi: B \rightarrow B[U]$. Then $\phi$ is said to be an {\it exponential map} on $B$, if the following conditions are satisfied:
 	
 \begin{itemize}
 	\item [\rm (i)] $\epsilon_{0} \phi_{U} =id_{B}$, where $\epsilon_{0}: B[U] \rightarrow B$ is the evaluation map at $U=0$.
 	
 	\item [\rm (ii)]  $\phi_{V} \phi_{U}=\phi_{U+V}$, where $\phi_{V}:B \rightarrow B[V]$ is extended to a $k$-algebra homomorphism $\phi_{V}:B[U] \rightarrow B[U,V]$, by setting $\phi_{V}(U)=U$.
 \end{itemize}

 	The ring of invariants of $\phi$ is a subring of $B$ defined as follows:
 	$$
 	B^{\phi}:=\left\{b \in B \, | \phi(b)=b  \right\}.
 	$$
 If $B^{\phi}\neq B$, then $\phi$ is said to be {\it non-trivial}.
 	Let EXP($B$) denote the set of all exponential maps on $B$. The {\it Makar-Limanov invariant} of $B$ is defined to be
 	$$
 	\ml(B):= \bigcap_{\phi \in \text{EXP}(B)} B^{\phi}
 	$$
 	and the {\it Derksen invariant} is a subring of $B$ defined as
 	$$
 	\dk(B):=k\left[ b \in B^{\phi} \,|\, \phi \in \text{EXP}(B)\, \text{and}\, \, B^{\phi} \subsetneqq B   \right].
 	$$
 
 	Next we record some useful results on exponential maps. The following two lemmas can be found in \cite[Chapter I]{miya}, \cite{cra} and \cite{inv}.
 	
 	\begin{lem}\thlabel{prope}
 	Let $B$ be an affine domain over $k$ and $\phi$ be a non-trivial exponential map on $B$. Then the following statements hold:
 	
 	\begin{itemize}

 		\item[\rm(i)] $B^{\phi}$ is a factorially closed subring of $B$, i.e., for any non-zero $a,b \in B$, if $ab \in B^{\phi}$, then $a, b \in B^{\phi}$. In particular, $B^{\phi}$ is algebraically closed in $B$. 
 		
 		\item[\rm(ii)] $\td_k B^{\phi}=\td_k B-1$.

 		\item[\rm(iii)]  For a multiplicatively closed subset $S$ of $B^{\phi} \setminus \{0\}$, $\phi$ induces a non-trivial exponential map $S^{-1}\phi$ on $S^{-1}B$ such that $(S^{-1}B)^{S^{-1}\phi}=S^{-1}(B^{\phi})$.
 	\end{itemize}

 \end{lem}

 	\begin{lem}\label{poly}
 		Let $k$ be a field and $B=k^{[n]}$. Then $\dk(B)=B$ for $n\ge 2$ and $\ml(B)=k$.
 	\end{lem}

 Next we define {\it proper} and {\it admissible $\mathbb{Z}$-filtration} on an affine domain. 
  
   \medskip
 \noindent
 {\bf Definition.}
  	Let $k$ be a field and $B$ be an affine $k$-domain. A collection $\{B_{n}\,\,| \, n \in \mathbb{Z}\}$ of $k$-linear subspaces of $B$ is said to be a {\it proper  $\mathbb{Z}$-filtration} if
  	
  	\begin{itemize}
  		\item [\rm (i)]  $B_{n} \subseteq B_{n+1}$, for every $n \in \mathbb{Z}$.
  		
  		\item [\rm (ii)] $B= \bigcup_{n \in \mathbb{Z}} B_{n}$.
  		
  		\item [\rm (iii)]  $\bigcap_{n \in \mathbb{Z}} B_{n}=\{0\}$.
  		
  		\item [\rm (iv)] $(B_{n} \setminus B_{n-1}).(B_{m}\setminus B_{m-1}) \subseteq B_{m+n} \setminus B_{m+n-1}$ for all $m,n \in \mathbb{Z}$.
  	\end{itemize}

A proper $\mathbb{Z}$-filtration $\{B_{n}\}_{n \in \mathbb{Z}}$ of $B$ is said to be {\it admissible} if there is 
a finite generating set $\Gamma$ of $B$ such that for each $n \in \mathbb{Z}$, every element in $B_{n}$ can be
 written as a finite sum of monomials from $B_{n} \cap k[\Gamma]$.

A proper $\mathbb{Z}$-filtration $\{B_{n}\}_{n \in \mathbb{Z}}$ of $B$ defines an associated graded domain defined by 
$$
gr(B):= \bigoplus_{n \in \mathbb{Z}} \frac{B_{n}}{B_{n-1}}.
$$
	It also defines the natural map $\rho: B \rightarrow gr(B)$ such that $\rho(b)=b+B_{n-1}$, if $b \in B_{n} \setminus B_{n-1}$.
	
	We now record a result on homogenization of exponential maps due to Derksen {\it{et\, al.}} \cite{dom}. The following version can be found in \cite[Theorem 2.6]{cra}.
	
	\begin{thm}\thlabel{dhm}
		Let $B$ be an affine domain over a field $k$ with an admissible proper $\mathbb{Z}$-filtration
		and $gr(B)$ be the induced $\mathbb{Z}$-graded domain. Let $\phi$ be a non-trivial exponential map on $B$. Then $\phi$
		induces a non-trivial homogeneous exponential map $\overline{\phi}$ on $gr(B)$ such that $\rho(B^{\phi}) \subseteq gr(B)^{\overline{\phi}}.$
		
	\end{thm}
	
	We quote below a criterion for flatness from \cite[(20.G)]{mat}.
	\begin{lem}\thlabel{flat}
		Let $R \rightarrow C \rightarrow D$ be local homomorphisms of Noetherian local rings, $\kappa$  the residue field of $R$ and $M$ a finite $D$ module. Suppose $C$ is $R$-flat. Then the following statements are equivalent:
		\begin{itemize}
			\item [\rm(i)] $M$ is $C$-flat.
			
			\item[\rm(ii)]  $M$ is $R$-flat and $M \otimes_{R} \kappa$ is $C \otimes_{R} \kappa$-flat.
		\end{itemize}
		\end{lem}

 The following is a well known result (\cite{AEH}).

\begin{thm}\thlabel{aeh}
	Let $k$ be a field and $R$ be a normal domain such that $k \subset R \subset k^{[n]}$. If $\td_{k} R=1$, then $R = k^{[1]}$.
\end{thm}

Next we state a version of the Russell-Sathaye criterion \cite[Theorem 2.3.1]{rs}, as 
presented in \cite[Theorem 2.6]{BD}.

\begin{thm}\thlabel{rs}
	Let $R \subset C$ be integral domains such that $C$ is a finitely generated $R$-algebra. 
Let $S$ be a multiplicatively closed subset of $R\setminus\{0\}$ generated by some prime elements of $R$ 
which remain prime in $C$. Suppose $S^{-1}C=(S^{-1}R)^{[1]}$ and, for every prime element $p \in S$, we have $pR=pC \cap R$ and $\frac{R}{pR}$ is algebraically closed in 
$\frac{C}{pC}$. Then $C=R^{[1]}$.
\end{thm}

 Next we note down a lemma which follows from \cite[Theorem 7]{dutta}.
 
 \begin{lem}\thlabel{sepco}
 	 Let $f \in k[Z,T]$ such that $L[Z,T]=L[f]^{[1]}$, for some separable field extension of $L$ of $k$. Then $k[Z,T]=k[f]^{[1]}$. 
 \end{lem}

We now recall a fundamental result on residual coordinates (\cite[Theorem 3.2]{bdc}). 
Recall that an integral domain $R$ with field of fractions $K$ is called a seminormal domain 
if for any $a \in K$, the conditions $a^2, a^3 \in R$ imply that $a\in R$.

 \begin{thm}\thlabel{bd}
	Let $R$ be a Noetherian domain such that either $\mathbb{Q} \subset R$ or $R$ is seminormal. 
Then the following statements are equivalent:
	\begin{itemize}
		\item [\rm(i)] $h \in R[X,Y]$ is a residual coordinate. 
		
		\item[\rm (ii)]  $R[X,Y]=R[h]^{[1]}$.
	\end{itemize}
\end{thm}

For the next three results, $A$ denotes the affine domain as in \eqref{A}.
The following result is proved in \cite[Lemma 3.3]{adv}.
	
	\begin{lem} \thlabel{subdk}
 $k[x_{1}, \ldots , x_{m}, z,t] \subseteq ~ \dk(A)$. 
	\end{lem}
	
	The following proposition of the second author is stated in \cite[Proposition 3.4(i)]{adv} under the hypothesis $\dk(A)=A$. 
	However, the proof uses only the consequence that $k[x_{1}, \ldots , x_{m},z,t] \subsetneqq \dk(A)$.
Below we quote the result under this modified hypothesis. 
	\begin{prop}\thlabel{dk}
	 Suppose that $k$ is infinite and $k[x_{1}, \ldots , x_{m},z,t] \subsetneqq \dk(A)$. 
	  Then there exist $Z_1, T_1 \in k[Z,T]$ and $a_0, a_1 \in k^{[1]}$ such that $k[Z,T]=k[Z_1,T_1]$ and $f(Z,T)=a_0(Z_1)+a_1(Z_1)T_1$.
\end{prop}

  \begin{rem}\thlabel{r2}
  \em{	For $m=1$, the hypothesis that $k$ is an infinite field in the above proposition can be dropped (cf. \cite[Proposition 3.7]{com}).}
  \end{rem}

The next result shows that if $f$ is a non-trivial line, then $\dk(A)=k[x_1,\ldots,x_m,z,t]$.

\begin{prop}\thlabel{r1}
Suppose that $k[x_1,\ldots,x_m,z,t] \subsetneqq \dk(A)$ and $k[Z,T]/(f)=k^{[1]}$. Then $k[Z,T]=k[f]^{[1]}$.
\end{prop}
\begin{proof}
	If $k$ is infinite, then the assertion follows from \cite[Proposition 3.4(ii)]{adv}. Now suppose $k$ is a finite field. Let $\overline{k}$ be an algebraic closure of $k$ and $\overline{A}=A \otimes_{k} \overline{k}$. Since $k[x_{1},\ldots,x_{m},z,t] \subsetneqq \dk(A)$, we have $\overline{k}[x_{1},\ldots,x_{m},z,t] \subsetneqq \dk(\overline{A})$. If $k[Z,T]/(f)=k^{[1]}$, then $\overline{k}[Z,T]/(f)=\overline{k}^{[1]}$. As $\overline{k}$ is infinite, we have $\overline{k}[Z,T]=\overline{k}[f]^{[1]}$ and hence by \thref{sepco}, we get that $k[Z,T]=k[f]^{[1]}$.
\end{proof}

In the rest of this section we will consider a Noetherian ring $R$ and 
note some $K$-theoretic aspects of $R$ (cf. \cite{bass}, \cite{bgl}).
Let $\mathscr{M}
(R)$ denote the category of finitely generated $R$-modules and
$\mathscr{P}(R)$ the category of finitely generated projective $R$-modules. Let $G_{0}(R)$ 
and $G_{1}(R)$ respectively denote the {\it Grothendieck group} and the {\it Whitehead group} 
of the category  $\mathscr{M}(R)$. Let $K_{0}(R)$ and $K_{1}(R)$ respectively denote the {\it Grothendieck group} 
and the {\it Whitehead group} of the category  $\mathscr{P}(R)$. For $i \geqslant 2$, the definitions of $G_{i}(R)$ and $K_{i}(R)$ 
can be found in (\cite{sr}, Chapters 4 and 5). The following results can be found in \cite[Propositions 5.15 and 5.16, Theorem 5.2]{sr}.

\begin{lem}\thlabel{lexact}
	Let $t$ be a regular element of $R$. Then the inclusion map $j:R \hookrightarrow R[t^{-1}]$ and the natural surjection map $\pi: R \rightarrow \frac{R}{tR}$ induce the following long exact sequence of groups:
	\begin{equation*}
		\begin{tikzcd}
			\arrow[r] &G_{i}(\frac{R}{tR}) \arrow[r, "\pi_{*}"] &G_{i}(R) \arrow[r, "j^{*}"] &G_{i}(R[t^{-1}]) \arrow[r] &\cdots \arrow[r] &G_{0}(R[t^{-1}]) \arrow[r] &0
		\end{tikzcd}
	\end{equation*}
	
\end{lem}

\begin{lem}\thlabel{fcom}
		Let $t$ be a regular element of $R$, $\phi: R \rightarrow C$ be a flat ring homomorphism and $u=\phi(t)$. Then we get the following commutative diagram:
		\begin{equation*}
			\begin{tikzcd}
			 \cdots\arrow[r] &G_{i}(\frac{R}{tR}) \arrow[r]\arrow[d,"\overline{\phi}^{*}"] & G_{i}(R)\arrow[r]\arrow[d,"\phi^{*}"] & G_{i}(R[t^{-1}])\arrow{r}\arrow[d,"(t^{-1}\phi)^{*}"] & G_{i-1}(\frac{R}{tR}) \arrow{r}\arrow[d, "\overline{\phi}^{*}"] &\cdots  \\
			 \cdots\arrow[r] &G_{i}(\frac{C}{uC}) \arrow[r] & G_{i}(C)\arrow[r] &G_{i}(C[u^{-1}]) \arrow{r}  & G_{i-1}(\frac{C}{uC}) \arrow{r} &\cdots ,
			\end{tikzcd}
		\end{equation*}
	where $\phi $ induces the vertical maps.
\end{lem}

\begin{lem}\thlabel{split}
For  an indeterminate $T$ over $R$, the following hold:
	
	\begin{itemize}
		\item [\emph{(a)}] For every $i \geqslant 0$, the maps $G_{i}(R) \rightarrow G_{i}(R[T])$, which are induced by $R \hookrightarrow R[T]$, are isomorphisms.
		
		\item[\emph{(b)}]   
		Let $j: R \rightarrow R[T,T^{-1}]$ be the inclusion map. Then for every $i$, $i \geqslant 1$, the following sequence is split exact. 
	$$
\begin{tikzcd}
0 \arrow[r] & G_{i}(R[T]) \arrow[r, "j^{*}"] &G_{i}(R[T,T^{-1}]) \arrow[r] & G_{i-1}(R) \arrow[r] &0
\end{tikzcd}.
$$
In particular, for $i=0$, $j^{*}$ is an isomorphism and for $i \geqslant 1$, $G_{i}(R[T,T^{-1}]) \cong G_{i}(R[T]) \oplus G_{i-1}(R) \cong G_{i}(R) \oplus G_{i-1}(R) $.
	\end{itemize}

\end{lem}

\begin{rem}\thlabel{rmk1}
	\emph{For a regular ring $R$, $G_{i}(R)=K_{i}(R)$, for every $i \geqslant 0$. In particular, $G_{1}(k[X])=K_{1}(k[X])=k^{*}$ and $G_{0}(k)=K_{0}(k)=\mathbb{Z}$. For $C_{m}:=k[X_{1}, \ldots , X_{m},X_{1}^{-1},\ldots,X_{m}^{-1}]$, $G_{0}(C_{m})=K_{0}(C_{m})=\mathbb{Z}$ (since every finitely generated projective module over $C_{m}$ is free) and by repeated application of \thref{split}(b), we get that $G_{1}(C_{m}) \cong k^{*} \oplus \mathbb{Z}^{m}$, for every $m\geqslant 1$. }
\end{rem}

\section{On Theorem B}

In this section we will prove an extended version of Theorem B.
For convenience, we first record an observation.                                                               

\begin{lem}\thlabel{rsco}
Let $R$ be an integral domain.
Let $E:=R[X_{1}, \ldots, X_{m}, T]$, $C= E[Z,Y]$, $g \in E[Z]$  and 
$G=X_{1}^{r_{1}} \cdots X_{m}^{r_{m}}Y- X_{1} \cdots X_{m}g+Z \in C$.
Then $C=E[G]^{[1]}$. 
\end{lem}

\begin{proof}

Let $D=E[G] \subseteq C$. Let $S$ be the multiplicatively closed subset of $D\setminus \{0\}$, 
which is generated by $X_{1},\ldots,X_{m}$. From the expression of $G$, it is clear that $S^{-1}C=(S^{-1}D)[Z]$.
Again $\frac{C}{X_{i}C}=(\frac{D}{X_{i}D})^{[1]}$ for every $i, 1 \leqslant i \leqslant m$. 
Therefore, by \thref{rs}, we get that $C=D^{[1]}$. 
\end{proof}

We now fix some notation. Throughout this section, 
 $A$ will denote the ring as in (\ref{A}), i.e., 
 	\begin{equation}\label{AG}
 	A= \frac{k\left[ X_{1}, \ldots , X_{m}, Y,Z,T\right]}{(X_{1}^{r_{1}} \cdots X_{m}^{r_{m}}Y -F(X_{1}, \ldots , X_{m}, Z,T))},\,\ r_i > 1 \text{~for all~} i, 1 \leqslant i \leqslant m,
 \end{equation}
where  $f(Z,T):=F(0, \ldots, 0, Z,T) \neq ~0$. Let $G:= X_{1}^{r_{1}} \cdots X_{m}^{r_{m}}Y -F$. 
Note that $A$ is an integral domain. 
Recall that ${x_{1}, \ldots, x_{m}, y, z,t}$ denote the images of 
${X_{1}, \ldots , X_{m},Y,Z,T}$ in $\mathbf{A}$.
$E$ and $B$ will denote the following subrings of $A$:
$$
{\mathbf{E=k[x_{1},\ldots,x_{m}]} \text{~and~} \mathbf{B=k[x_1,\ldots,x_m,z,t]}.}
$$ 
 Note that
$$
B=k[x_{1}, \ldots, x_{m},z,t] \hookrightarrow A \hookrightarrow B[x_{1}^{-1},\ldots,x_{m}^{-1}].  
$$
Fix $(e_{1}, \ldots, e_{m}) \in \mathbb{Z}^{m}$. The ring $B[x_{1}^{-1}, \ldots, x_{m}^{-1}]$ 
can be given the following $\mathbb{Z}$-graded structure:
$$
B[x_{1}^{-1},\ldots,x_{m}^{-1}] = \bigoplus_{n \in \mathbb{Z}} B_{n},
$$
where $B_{n} = \bigoplus_{(i_{1},\ldots,i_{m}) \in \mathbb{Z}^{m}, e_{1}i_{1}+\cdots+e_{m}i_{m}=n} k[z,t] x_{1}^{i_{1}} \cdots x_{m}^{i_{m}}$. Now every $b \in B[x_{1}^{-1},\ldots,x_{m}^{-1}]$ can be written uniquely as $b=\sum_{i=d_{l}}^{d_{h}} b_{i}$, where $b_{i} \in B_{i}$, $d_{l}, d_{h}$ are some integers and $b_{d_{l}}, b_{d_{h}} \neq 0$. If $b \in B$, then each $b_{i} \in B$. We call $d_{h}$ the degree of $b$ and hence $b_{d_{h}}$ is the highest degree homogeneous summand of $b$.
For every $n \in \mathbb{Z}$, if $A_{n} = \bigoplus_{i \leqslant n} B_{n} \cap A$, then $\{A_{n}\}_{n \in \mathbb{Z}}$ defines a $\mathbb{Z}$-filtration on $A$. For every $j, 1 \leqslant j \leqslant m$, $x_{j} \in A_{e_{j}} \setminus A_{e_{j}-1}$. 
If $d$ denotes the degree of $F(x_{1}, \ldots , x_{m},z,t)$, then $y \in A_{\l} \setminus A_{\l-1}$, where $\l=d-(r_{1}e_{1}+\cdots+r_{m}e_{m})$. 

\smallskip
With respect to the notation as above, we quote two lemmas \cite[Lemmas 3.1, 3.2]{adv}.

\begin{lem}\thlabel{ad}
	The $k$-linear subspaces $\{A_{n}\}_{n \in \mathbb{Z}}$ define a proper 
admissible $\mathbb{Z}$-filtration on $A$ with the generating set $\Gamma =\{x_{1}, \ldots , x_{m},y,z,t\}$, and the associated graded ring $gr(A) =\bigoplus_{n \in \mathbb{Z}} \frac{A_{n}}{A_{n-1}}$ is generated by the image of $\Gamma$ in $gr(A)$. 
\end{lem}   

\begin{lem}\thlabel{gr}
	Let $d$ denote the degree of $F(x_{1}, \ldots, x_{m},z,t)$ and $F_{d}$ denote
 the highest degree homogeneous summand of $F$. Suppose that, for each $j$, $1 \leqslant j \leqslant m$,
 $x_{j} \nmid F_{d}$. Then the associated graded ring $gr(A)$ is isomorphic to 
	$$
	\frac{k[X_{1}, \ldots , X_{m},Y,Z,T]}{ \left(X_{1}^{r_{1}}\cdots X_{m}^{r_{m}}Y- F_{d}(X_{1}, \ldots , X_{m},Z,T)\right)}
	$$
	as $k$-algebras.
\end{lem}

The next result describes $\ml(A)$ when $\dk(A)$ is exactly equal to $B(=k[x_{1}, \ldots , x_{m},z,t])$. 

\begin{prop}\thlabel{ml}
	Let $A$ be the affine domain as in \eqref{AG}. Then the following hold:
	
	\begin{itemize}
		\item [\rm (a)] Suppose, for every $i \in \{1,\ldots,m\}$, $x_i \notin A^{*}$, and $F \notin k[X_1,\ldots,X_m]$. 
Then $\ml(A) \subseteq E(= k[x_{1},\ldots,x_{m}])$.
		
		\item[\rm (b)]  If $\dk(A)=B$, then $\ml(A)=E$.
	\end{itemize}
\end{prop}

\begin{proof}
	(a) Since $F \notin k[X_1,\ldots,X_m]$, without loss of generality, suppose $\deg_T F > 0$.
	Consider the map $\phi_{1}: A \rightarrow A[U]$ defined as follows:
	$$
	\phi_{1}(x_{i})=x_{i}\,(1 \leqslant i \leqslant m),\,\,\, \phi_{1}(z)=z,\,\,\, \phi_{1}(t)=t+ x_{1}^{r_{1}}\cdots x_{m}^{r_{m}}U,
	$$
	and
	$$
	\phi_{1}(y)=\frac{F(x_{1}, \ldots , x_{m}, z, t+x_{1}^{r_{1}} \cdots x_{m}^{r_{m}}U)}{x_{1}^{r_{1}} \cdots x_{m}^{r_{m}}} = y+ Uv(x_{1}, \ldots, x_{m},z,t,U),
	$$
	for some $v\in k[x_{1},\ldots,x_{m},z,t,U]$. It is easy to see that $\phi_{1}\in \text{EXP}(A)$. Now $$k[x_{1}, \ldots, x_{m},z] \subseteq A^{\phi_{1}} \subseteq A \subseteq k[x_{1}, \ldots, x_{m},(x_{1}\cdots x_{m})^{-1},z,t].$$ 
	Since $\deg_tF > 0$, and for every $i$, $1 \leqslant i \leqslant m$, $x_i \notin A^{*}$, it follows that $$
	A \,\cap\, k[x_1,\ldots,x_m,(x_1 \cdots x_m)^{-1},z ]=k[x_1,\ldots,x_m,z].
	$$ 
	Therefore, $k[x_{1}, \ldots, x_{m},z]$ is algebraically closed in $A$.
	Also $\td_{k}(k[x_{1}, \ldots, x_{m},z])= \td_{k}(A^{\phi_{1}})=m+1$ (cf. \thref{prope}(ii)). 
        Hence $A^{\phi_{1}}=k[x_{1},\ldots,x_{m},z]$.
	
	Again consider the map $\phi_{2}: A \rightarrow A[U]$ defined as follows:
	$$
	\phi_{2}(x_{i})=x_{i}\,(1 \leqslant i \leqslant m),\,\,\, \phi_{2}(t)=t,\,\,\, \phi_{2}(z)=z+ x_{1}^{r_{1}}\cdots x_{m}^{r_{m}}U,
	$$
	and
	$$
	\phi_{2}(y)=\frac{F(x_{1}, \ldots , x_{m}, z+x_{1}^{r_{1}} \cdots x_{m}^{r_{m}}U, t)}{x_{1}^{r_{1}} \cdots x_{m}^{r_{m}}} = y+ Uw(x_{1}, \ldots, x_{m},z,t,U),
	$$
	for some $w \in k[x_{1},\ldots,x_{m},z,t,U]$. It follows that $\phi_2 \in \text{EXP}(A)$. Clearly 
	$$
	k[x_{1}, \ldots, x_{m},t] \subseteq A^{\phi_{2}} \subseteq k[x_1,\ldots,x_m,(x_1 \cdots x_m)^{-1},t].
	$$
	Therefore, $\ml(A) \subseteq A^{\phi_{1}} \cap A^{\phi_{2}} \subseteq k[x_{1}, \ldots , x_{m}]=E$.
	
	\medskip
	
	\noindent
	(b)  Suppose $\dk(A)=B$. Note that for every $i$, $1\leqslant i \leqslant m$, 
$x_i \notin A^{*}$, otherwise $x_i^{-1} \in \dk(A)$. Further, if either $\deg_TF=0$  
or $\deg_ZF=0$, then either $y\in A^{\phi_{1}}$ or $y\in A^{\phi_{2}}$ respectively, 
where $\phi_{1},\phi_{2} \in \text{EXP}(A)$ are as defined in part (a). 
This contradicts our assumption that $\dk(A)=B$. 
Therefore, $F \notin k[X_1,\ldots,X_m]$, and hence by part (a), it is clear that $\ml(A) \subseteq E$. 
	
	Let $\phi \in \text{EXP}(A)$. We show that $E \subseteq A^{\phi}$. Now $\td_{k} A^{\phi}=m+1$ (cf. \thref{prope}(ii)) and $A^{\phi} \subseteq k[x_{1}, \ldots, x_{m},z,t]$. Suppose $\{ f_{1}, \ldots, f_{m+1}\}$ is an algebraically independent set of elements in $A^{\phi}$. We fix some $ j  \in \{1, \ldots, m\}$ and let
	\begin{equation}\label{f_{i}}
		f_{i} =g_{i}(x_{1}, \ldots, x_{j-1},x_{j+1}, \ldots, x_{m},z,t)+ x_{j} h_{i}(x_{1}, \ldots, x_{m},z,t)
	\end{equation}
	for every $i \in \{1,\ldots,m\}$. We show that the set $\{g_{1},\ldots, g_{m+1}\}$ is algebraically dependent. 
	
	Suppose not.
	We consider the $\mathbb{Z}$-filtration on $A$ induced by $(0,\ldots,0,-1,0,\ldots,0) \in \mathbb{Z}^m$, where the $j$-th entry is $-1$. If 
$f_{id}$ denotes the highest degree homogeneous summand  of $f_{i}$, 
	then from \eqref{f_{i}}, we get that $f_{id}=g_{i}$. By \thref{dhm}, $\phi$ will induce a non-trivial exponential map $\phi_{j}$ on the associated graded ring $A_{j}$ and 
	$\{g_{i}\,\,|\,\, 1 \leqslant i \leqslant m+1 \} \subseteq A_{j}^{\phi_{j}}$. By \thref{gr},
	\begin{equation}\label{A_j}
		A_{j} \cong \frac{k[X_{1}, \ldots, X_{m}, Y,Z,T]}{\left(X_{1}^{r_{1}}\cdots X_{m}^{r_{m}}Y- F(X_{1}, \ldots X_{j-1},0,X_{j+1},\ldots, X_{m},Z,T) \right) }.
	\end{equation}
	 Since $\{g_{i}\,\,|\,\, 1 \leqslant i \leqslant m+1 \}$ is algebraically independent, $k[x_{1},\ldots, x_{j-1},x_{j+1},\ldots, x_{m},z,t]$ is algebraic over $k[g_{i} \,\,|\,\, 1 \leqslant i \leqslant m+1]$. As 
	$k[g_{i}\,\,| \,\, 1 \leqslant i \leqslant m+1] \subseteq A_{j}^{\phi_{j}}$ and $A_{j}^{\phi_{j}}$ is algebraically closed, we have $k[x_{1}, \ldots, x_{j-1},x_{j+1}, \ldots , x_{m},z,t] \subseteq A_{j}^{\phi_{j}}$. Therefore, from \eqref{A_j}, we get that $x_{j}, y \in A_{j}^{\phi_{j}}$, which contradicts that $\phi_{j}$ is non-trivial. Thus, $\{g_{i}\,\,|\,\, 1 \leqslant i \leqslant m+1 \}$ is algebraically dependent. Hence there exists a polynomial $P \in k^{[m+1]}$ such that 
	$$
	P(g_{1}, \ldots, g_{m+1})=0.
	$$
	Therefore, from \eqref{f_{i}}, we get that there exists $H\in k[x_{1}, \ldots, x_{m},z,t]$ 
such that $x_{j}H=P(f_{1}, \ldots, f_{m+1}) \in A^{\phi}$. 
	As $A^{\phi}$ is factorially closed (cf. \thref{prope}(i)), we have $x_{j} \in A^{\phi}$.
 Since $j$ is arbitrarily chosen from $\{1,\ldots,m\}$, we have $E=k[x_{1},\ldots,x_{m}] \subseteq A^{\phi}$.
 As $\phi \in \text{EXP}(A)$ is arbitrary, we have $E \subseteq \ml(A)$. Therefore, $\ml(A)=E$.
\end{proof}

	\begin{rem}\thlabel{r3}
\em{From \thref{dk}, it follows that if $k$ is an infinite field and there is no system of coordinates 
$\{Z_1,T_1\}$ of $k[Z,T]$ such that $f(Z,T)=a_0(Z_1)+a_1(Z_1)T_1$, 
then $\dk(A)=B$ and hence $\ml(A)=E$.}
\end{rem} 
 
Our next proposition gives equivalent conditions for the ring   $A$  to be a UFD.
\begin{prop}\thlabel{ufd2}
		The following statements are equivalent:
	\begin{enumerate}
		\item[\rm(i)] $A$ is a UFD.
		\item[\rm(ii)] For each $j, 1\leqslant j \leqslant m$, either $x_{j}$ is prime in $A$ or  $x_j \in A^{*}$.
		\item[\rm(iii)] For each $j, 1 \leqslant j \leqslant m$, $F_j:=F(X_1,\ldots, X_{j-1},0,X_{j+1},\ldots,X_m,Z,T)$ is either an irreducible element in $k[X_1,\ldots,X_{j-1},X_{j+1},\ldots,X_m,Z,T]$ or  $F_j \in k^{*}$.
\end{enumerate}
	In particular, if $	F(X_{1}, \ldots, X_{m}, Z,T)= f(Z,T)+ (X_{1} \cdots X_{m})g$, for some $g \in k[X_1,\ldots,X_m,Z,T]$, then the following statements are equivalent:
	\begin{enumerate}
		\item[$\rm( i^{\prime})$] $A$ is a UFD.
		\item[$\rm(ii^{\prime})$] For each $j, 1\leqslant j \leqslant m$, either $x_{j}$ is prime in $A$ or  $x_j \in A^{*}$.
		\item[$\rm(iii^{\prime})$]  $f(Z,T)$ is either an irreducible element in $k[Z,T]$ or $f(Z,T) \in k^{*}$.
	\end{enumerate}
\end{prop}

\begin{proof}
	The equivalence of the statements for $m=1$ has been done in \cite[Lemma 3.1]{com}. We will show this result for $m>1$.
	
	\smallskip
	\noindent
	$\rm (i) \Rightarrow (ii):$
	Note that $B \hookrightarrow A \hookrightarrow B[x_1^{-1},\ldots, x_m^{-1}]$. Suppose $x_j \notin A^{*}$ for some $j, 1 \leqslant j \leqslant m$. Since $A$ is a UFD, it is enough to show that $x_j$ is irreducible. Suppose $x_j= ab$ for some $a,b \in A$. If $a,b \in B$, then either $a \in B^{*}$ or $b \in B^{*}$, as $x_j$ is irreducible in $B$. Therefore, we can assume that at least one of them is not in $B$. Suppose $a \notin B$. Let $a= \frac{h_1}{x_1^{i_1}\cdots x_m^{i_m}}$ and $b= \frac{h_2}{x_1^{j_1}\cdots x_m^{j_m}}$, for some $h_1,h_2 \in B$ and $i_s,j_s \geqslant 0$, $1 \leqslant s \leqslant m$. Therefore, we have 
	\begin{equation}\label{h}
		h_1h_2= x_j(x_1^{i_1+j_1} \cdots x_m^{i_m+j_m}).
	\end{equation}
	As $a \notin B$, using \eqref{h}, without loss of generality, we can assume that  
	\begin{equation}\label{a1}
	a= \lambda \frac{x_1^{p_1} \cdots x_s^{p_s}}{x_{s+1}^{p_{s+1}} \cdots x_m^{p_m}},
   \end{equation}
	for some $\lambda \in k^{*}$ and $ s, 1 \leqslant s < m$, where $p_i \geqslant 0$ for $1 \leqslant i \leqslant s$, and $p_i > 0$ for $i \geqslant s+1$. If $p_i=0$ for every $i \leqslant s$, then $a \in A^{*}$, and we are done. If not, then without loss of generality, we may assume that $p_1>0$. Therefore, from \eqref{a1}, we have 
	$$
	x_1^{p_1} \cdots x_s^{p_s} \in  x_mA \cap B= \left( x_m, F \right)B.
	$$
	 But this is a contradiction as $f(Z,T)=F(0,\ldots,0,Z,T) \neq 0$.
	Therefore, we obtain that $x_j$ must be an irreducible element in $A$, hence prime in $A$.

	\smallskip
	\noindent
	$\rm (ii) \Leftrightarrow (iii):$
	For every $j, 1 \leqslant j \leqslant m$, we have
\begin{equation}\label{xj}
		\frac{A}{x_{j}A} \cong \frac{k[X_1,\ldots,X_{j-1},X_{j+1},\ldots,X_m,Z,T]}{(F_j)}.
\end{equation}
	Note that $x_j$ is either a prime element or a unit in $A$ if and only if $\frac{A}{x_{j}A}$ is an integral domain or a zero ring. Hence the equivalence follows from \eqref{xj}.
	
	\smallskip
	\noindent
	$\rm (ii) \Rightarrow (i):$ 
Without loss of generality we assume that $x_1,\ldots,x_i \in A^{*}$ and $x_{i+1},\ldots,x_m$ are primes in $A$
for some $i$, $1\leqslant i \leqslant m$.
	Since $A[x_{1}^{-1}, \ldots, x_{m}^{-1}]= A[x_{i+1}^{-1},\ldots,x_m^{-1}]=k[x_{1}, \ldots , x_{m}, x_{1}^{-1}, \ldots , x_{m}^{-1}]^{[2]}$ is a UFD, by Nagata's criterion for UFD (\cite[Theorem 20.2]{matr}), we obtain that $A$ is a UFD. 
\end{proof}

The next result gives a condition for $A$ to be flat over $E$.

\begin{lem}\thlabel{flat1}
  Suppose  $F(X_{1}, \ldots, X_{m}, Z,T)= f(Z,T)+ (X_{1} \cdots X_{m})g$, for some $g \in k[X_1,\ldots,X_m,Z,T]$. Then $A$ is a flat $E$-algebra.
\end{lem}

\begin{proof}
	Let $\q \in \Spec (A)$ and $\p =\q \cap A \in \Spec(E)$. Note that $A[x_1^{-1},\ldots,x_m^{-1}]=E[x_1^{-1},\ldots,x_m^{-1},z,t]$. Hence $A_{\q}$ is a flat $E_{\p}$ algebra if $x_i \notin \p$ for every $i, 1 \leqslant i \leqslant m$. Now    
	suppose $x_i \in \p$ for some $i$. Consider the following maps:
	
	$$
	k[x_i]_{(x_i)} \hookrightarrow E_{\p} \hookrightarrow A_{\q}.
	$$
	We observe that both $E_{\p}$ and $A_{\q}$ are flat over $k[x_i]_{(x_i)}$ and
	$$
	\frac{A}{x_iA} \cong \frac{(E/x_iE)[Y,Z,T]}{(f(Z,T))}.
	$$
	Since $\frac{k[Y,Z,T]}{(f(Z,T))}$ is a flat $k$-algebra, it follows that $\frac{A}{x_iA} \left(=\frac{k[Y,Z,T]}{(f(Z,T)) } \otimes_{k} \frac{E}{x_iE}\right)$ is a flat $\frac{E}{x_iE}$-algebra. Hence it follows that $\frac{A_{\q}}{x_iA_{\q}}$ is flat over $\frac{E_{\p}}{x_iE_{\p}}$ and therefore, by \thref{flat}, we get that $A_{\q}$ is flat over $E_{\p}$. Thus, $A$ is locally flat over $E$, and hence $A$ is flat over $E$.
\end{proof}

The next result gives some necessary and sufficient conditions for $A$ to be an affine fibration over $E$.

\begin{prop}\thlabel{line}
	    Suppose $F(X_{1}, \ldots, X_{m}, Z,T)= f(Z,T)+ (X_{1} \cdots X_{m})g$, for some $g \in k[X_1,\ldots,X_m,Z,T]$.
	    Then the following statements are equivalent:
		
		\begin{enumerate}
			\item [\rm(i)] $A$ is an $\mathbb{A}^{2}$-fibration over $E$. 
			
			\item[\rm(ii)] $\frac{A}{(x_{1}, \ldots , x_{m})A}=k^{[2]}$.
			
			\item[\rm(iii)] $f(Z,T)$ is a line in $k[Z,T]$, i.e., $\frac{k[Z,T]}{(f(Z,T))}=k^{[1]}$.
		\end{enumerate}
\end{prop}

\begin{proof}
	
	\noindent
	$\rm(i) \Rightarrow \rm(ii):$
	 Since $A$ is an $\mathbb{A}^{2}$-fibration over $E$, for every $\p \in \Spec(E)$, we have
	$$
	A\otimes_{E} \frac{E_{\p}}{\p E_{\p}} = \left(\frac{E_{\p}}{\p E_{\p}}
	\right)^{[2]}
.	$$
	Hence for $\p=(x_{1}, \ldots , x_{m})E$, we get $\frac{A}{(x_{1},\ldots,x_{m})A}=k^{[2]}.$
	
	\smallskip
	\noindent
	$\rm(ii) \Rightarrow \rm(iii):$
		Since
		$$
		k \hookrightarrow \frac{k[Z,T]}{(f(Z,T))}\hookrightarrow \frac{A}{(x_{1}, \ldots , x_{m})A}= \left(\frac{k[Z,T]}{(f(Z,T))}\right)^{[1]}=k^{[2]},
		$$
		we obtain that $\frac{k[Z,T]}{(f(Z,T))}$ is a one dimensional normal domain. Hence, $f(Z,T)$ is a line in $k[Z,T]$ by  \thref{aeh}.
		
		\smallskip
		\noindent
		$\rm(iii) \Rightarrow \rm(i):$ 
		 By \thref{flat1}, $A$ is a flat $E$-algebra. Let $\p \in \Spec(E) $ and $A_{\p}$ denotes the localisation of the ring $A$ with respect to the multiplicatively closed set $E\setminus \p$. We now show $A$ is an $\mathbb{A}^{2}$-fibration over $E$.
		
		\smallskip
		 \noindent
		{\it Case} 1:
		If $x_i \notin \p$ for every $i=1, \ldots ,m$, then 
		$A_{\p}=E_{\p}[z,t]$. Hence, 
		
		\begin{equation}\label{lf1}
		\frac{A_{\p}}{\p A_{\p}}= A \otimes_{E} \frac{E_{\p}}{\p E_{\p}}= \left(\frac{E_{\p}}{\p E_{\p}}\right)^{[2]}
		.
	\end{equation}
		
		\noindent
		{\it Case} 2:
		Suppose $x_i \in \p$ for some $i, 1 \leqslant i \leqslant m$. Since $f(Z,T)$ is a line in $k[Z,T]$, we have
		
		$$
		 \frac{A}{x_iA} 
		 = \frac{(E/x_iE)[Y,Z,T]}{(f(Z,T))}= \left( \frac{E}{x_iE}\right)^{[2]}  .
		$$

Hence, $A \otimes_{E} \kappa(\p)= \kappa(\p)^{[2]}$, where  $\kappa(\p)= \frac{E_{\p}}{\p E_{\p}}$.

Therefore, by the above two cases, we obtain that $A$ is an $\mathbb{A}^{2}$-fibration over $E$. 		
\end{proof}

 The following lemma may be known but in the absence of a ready reference, a proof is included here.

\begin{lem}\thlabel{g0}
	Let $E=k[x_1,\ldots,x_m]$, $u=x_1 \cdots x_m$ and $R=\frac{E}{uE}$. Then $G_0(R) \neq 0$.
\end{lem}

\begin{proof}
		By \thref{lexact} the inclusion $E \hookrightarrow E[u^{-1}]$ induces the exact sequence:
	\begin{equation}\label{9}
		\begin{tikzcd}
			G_{1}(E) \arrow[r] &G_{1}(E[u^{-1}])\arrow[r] & G_{0}(R) \arrow[r] &G_{0}(E) \arrow[r, "h"] &G_{0}(E[u^{-1}]) \arrow[r] &0.
		\end{tikzcd}
	\end{equation}
	By \thref{rmk1} and repeated application of \thref{split}, we see that the map 
	\begin{equation*}
		\begin{tikzcd}
			G_{1}(E) \cong k^{*} \arrow[r, hook] &G_{1}(E[u^{-1}]) \cong k^{*} \oplus \mathbb{Z}^{m}
		\end{tikzcd}
	\end{equation*}
	is a split inclusion and hence can not be surjective.
	  Therefore, $G_{0}(R) \neq 0$ by \eqref{9}.
\end{proof}

We now prove our main result.

\begin{thm}\thlabel{equiv}
Suppose $	F(X_{1}, \ldots, X_{m}, Z,T)= f(Z,T)+ (X_{1} \cdots X_{m})g$, for some $g \in k[X_1,\ldots,X_m,Z,T]$.
 Then the following statements are equivalent:
	\begin{enumerate}
		\item [\rm(i)] $k[X_{1}, \ldots , X_{m},Y,Z,T]=k[X_{1}, \ldots, X_{m},G]^{[2]}$.
		
			\item[\rm(ii)]  $k[X_{1}, \ldots , X_{m},Y,Z,T]=k[G]^{[m+2]}$.
		
		\item[\rm(iii)] $A=k[x_{1}, \ldots , x_{m}]^{[2]}$.
		
		\item[\rm(iv)] $A=k^{[m+2]}$.
		
		\item[\rm(v)] $k[Z,T]=k[f(Z,T)]^{[1]}$.

	 \item[\rm(vi)] $A^{[l]}=k^{[l+m+2]}$ for some $l \geqslant 0$ and $k[x_{1},\ldots,x_{m},z,t]   \subsetneqq \dk(A)$.
		
		\item[\rm(vii)]  $A$ is an $\mathbb{A}^{2}$-fibration over $k[x_{1}, \ldots , x_{m}]$ and $k[x_{1}, \ldots , x_{m},z,t] \subsetneqq \dk(A)$.
		
		\item [\rm(viii)] $f(Z,T)$ is a line in $k[Z,T]$ and $k[x_{1}, \ldots , x_{m},z,t] \subsetneqq \dk(A)$.

			\item[\rm(ix)]  $f(Z,T)$ is a line in $k[Z,T]$ and $\ml(A)=k$.

		\item[\rm(x)] 	$A^{[l]}=k^{[l+m+2]}$ for $l \geqslant 0$ and $\ml(A)=k$.
		
		\item [\rm(xi)] $A$ is an $\mathbb{A}^{2}$-fibration over $k[x_{1},\ldots,x_{m}]$ and $\ml(A)=k$.

	\end{enumerate}
	
\end{thm}

\begin{proof}
 Note that $\rm (i) \Rightarrow (ii) \Rightarrow (iv) \Rightarrow (vi)$, $\rm (i) \Rightarrow (iii) \Rightarrow (iv)$, $\rm(iii) \Rightarrow (vii)$, $\rm (iii) \Rightarrow (x)$ and $\rm (iii) \Rightarrow (xi)$ follow trivially. 
 
 \medskip
 \noindent 
$\rm (vi) \Rightarrow (v) :$ 
 We first assume that $k$ is algebraically closed. Since $k[x_{1},\ldots,x_{m},z,t] \subsetneqq \dk(A)$, by \thref{dk}, we may assume that 
$$
f(Z,T)=a_0(Z)+a_1(Z)T.
$$

 Suppose $a_{1}(Z)=0$, i.e., $f(Z,T)=a_{0}(Z)$. As $A$ is a UFD, by \thref{ufd2}, we obtain that $a_0(Z)$ is irreducible in $k[Z,T]$ or $a_0(Z) \in k^{*}$. If $a_0(Z) \in k^{*}$, then for every $i$, $1 \leqslant i \leqslant m$, $x_{i}\in A^{*}$. This contradicts the fact that $A^{*}=(A^{[l]})^{*}=(k^{[m+l+2]})^{*}=k^{*}$. Therefore, $a_0(Z)$ is irreducible in $k[Z,T]$ and hence a linear polynomial as $k$ is algebraically closed. Thus $f$ is a coordinate of $k[Z,T]$.

 Now suppose $a_1(Z) \neq 0$. We show that $a_{1}(Z) \in k^{*}$. Note that
${\rm gcd}(a_0(Z),a_1(Z))=1$ in $k[Z]$, as $f(Z,T)$ is irreducible in $k[Z,T]$.

 Consider the inclusions
\begin{tikzcd}
	k \arrow[r, hook, "\alpha"] & E \arrow[r,hook, "\beta"] &A \arrow[r, hook, "\gamma"] & A^{[l]}
\end{tikzcd}. Since $E=k^{[m]}$ and $A^{[l]}=k^{[m+l+2]}$, by \thref{split}(a), the inclusions $\alpha, \gamma$ and $\gamma\beta \alpha$ induce isomorphisms
$$
\begin{tikzcd}
	G_{i}(k) \arrow[r, "G_{i}(\alpha)"', "\cong" ] & G_{i}(E),
\end{tikzcd} 
\begin{tikzcd}
	G_i(A) \arrow[r, "G_i(\gamma)"',"\cong"] & G_i(A^{[l]})
\end{tikzcd}
\text{~and~} 
\begin{tikzcd}
	G_{i}(k) \arrow[r, "G_{i}(\gamma\beta \alpha)"', "\cong"] & G_{i}(A^{[l]}),
\end{tikzcd}
$$
respectively. Hence we get that $\beta$ induces an isomorphism 
\begin{tikzcd}
	G_{i}(E) \arrow[r, "G_{i}(\beta)"', "\cong"] &G_{i}(A),
\end{tikzcd} for every $i \geqslant 0$.
Let $u=x_1 \cdots x_m \in E$. Note that $A[u^{-1}]=E[u^{-1},z,t]$. Therefore, by \thref{split}(a), the inclusion $E[u^{-1}] \hookrightarrow A[u^{-1}]$ induces isomorphisms $G_{i}(E[u^{-1}]) \xrightarrow{\cong} G_{i}(A[u^{-1}])$ for $i \geqslant 0$. By \thref{flat1}, the inclusion map
\begin{tikzcd} 
	\beta : E \arrow[r,hook] &A 
\end{tikzcd} 
is a flat map. 
Therefore, by \thref{fcom}, $\beta : E \hookrightarrow A$ induces the following commutative diagram for $i \geqslant 1$:
$$
\begin{tikzcd}
	G_{i}(E) \arrow{r} \arrow[d, "G_i(\beta)"',"\cong"] & G_{i}(E[u^{-1}])\arrow{r}\arrow[d,"\cong"] & G_{i-1}(\frac{E}{uE}) \arrow{r}\arrow{d} & G_{i-1}(E)\arrow{r}\arrow[d,"\cong"] & G_{i-1}(E[u^{-1}]) \arrow[d,"\cong"] \\
	G_{i}(A) \arrow{r} & G_{i}(A[u^{-1}])\arrow{r} & G_{i-1}(\frac{A}{uA}) \arrow{r} & G_{i-1}(A) \arrow{r} & G_{i-1}(A[u^{-1}]) .
\end{tikzcd}
$$
From the above diagram, applying the Five Lemma we obtain that the map 
\begin{tikzcd}
	\frac{E}{uE} \arrow[r,"\overline{\beta}"]  &\frac{A}{uA},
\end{tikzcd} induced by 
\begin{tikzcd}
	E\arrow[r,hook, "\beta"] &A
\end{tikzcd}, induces isomorphism of groups
\begin{equation}\label{4}
	\begin{tikzcd}
		G_{i}(\frac{E}{uE}) \arrow[r,"\cong", "G_{i}(\overline{\beta})"'] &G_{i}(\frac{A}{uA}),
	\end{tikzcd}
\end{equation}
for every $i \geqslant 0$. Let $R=\frac{E}{uE}$. Now using the structure of $f(Z,T)$ we get an isomorphism as follows
$$
\begin{tikzcd}
	\frac{A}{uA} \arrow[r, "\cong", "\eta"'] &R\left[Y,Z,\frac{1}{a_{1}(Z)}\right]
\end{tikzcd}.
$$
 Further, note that $\overline{\beta}$ factors through the following maps 
$$
\overline{\beta}: R \xrightarrow[]{\gamma_1} R[Z] \xrightarrow[]{\gamma_2} R\left[Z, \frac{1}{a_{1}(Z)} \right] \xrightarrow[]{\gamma_3}  R\left[Y, Z, \frac{1}{a_{1}(Z)} \right] \xrightarrow[]{\eta^{-1}}  A/uA.
$$
Since $\overline{\beta},\gamma_1,\gamma_3,\eta^{-1}$ induce isomorphisms of $G_i$-groups for $i \geqslant 0$, we obtain that $\gamma_2$ induces an isomorphisms of groups
\begin{equation}\label{5}
	\begin{tikzcd}
		G_{i}(R[Z]) \arrow[r,"\cong", "G_{i}(\gamma_2)"'] &G_{i} \left(R\left[Z, \frac{1}{a_{1}(Z)} \right] \right),
	\end{tikzcd}
\end{equation}
for every $i \geqslant 0$.
Since $k$ is algebraically closed, if $a_1(Z) \notin k^{*}$, then we have
$a_{1}(Z)= \lambda\prod_{i=1}^{n}(Z-\lambda_i)^{m_i}$ where $\lambda \in k^*,\lambda_i \in k$, $m_i \geqslant 1$ and $\lambda_i \neq \lambda_{j}$ for $i \neq j$. Now we have 
$$
R\left[ Z, \frac{1}{a_{1}(Z)}\right]= R\left[Z, \frac{1}{Z-\lambda_1}, \ldots , \frac{1}{Z-\lambda_n} \right].
$$
Let $R_{i}=R\left[Z, \frac{1}{Z-\lambda_1}, \ldots, \frac{1}{Z-\lambda_i} \right]$, 
for $i \geqslant 1$ and $R_{0}=R[Z]$. For every $i \geqslant 1$, we have the map
 $R[Z] \hookrightarrow R_{i-1}$ is flat. Therefore, applying \thref{fcom}, 
we get the following commutative diagram for $j \geqslant 1$:
\begin{equation}\label{cd3}
\begin{tikzcd}
		G_{j}\left(\frac{R[Z]}{(Z-\lambda_i)}\right) \arrow[r]\arrow[d,"\mu_j" ] & G_{j}(R[Z])\arrow{r}\arrow[d] & G_{j}\left(R\left[Z,\frac{1}{Z-\lambda_i}\right]\right) \arrow{r}\arrow{d} & G_{j-1}\left(\frac{R[Z]}{(Z-\lambda_i)}\right)\arrow[r]\arrow[d,"\mu_{j-1}"] & \cdots \\
		G_{j}\left(\frac{R_{i-1}}{Z-\lambda_i}\right) \arrow[r] & G_{j}(R_{i-1})\arrow{r} & G_{j}\left(R_{i-1}\left[ \frac{1}{Z-\lambda_i}\right]\right) \arrow{r} & G_{j-1}\left(\frac{R_{i-1}}{Z-\lambda_i}\right) \arrow{r} & \cdots
	\end{tikzcd}
\end{equation}
	Now by \thref{split}(b) we get the following split exact sequence for each $j \geqslant 1$:
	 	$$\begin{tikzcd}
	 0 \arrow[r] & G_{j}(R[Z])\arrow{r} & G_{j}\left(R\left[Z,\frac{1}{Z-\lambda_i}\right]\right) \arrow{r} & G_{j-1}\left(\frac{R[Z]}{(Z-\lambda_i)}\right)\arrow[r] & 0 .
	 \end{tikzcd}$$
 	Since the inclusion $R[Z] \hookrightarrow R_{i-1}$ induces isomorphism of the rings $\frac{R[Z]}{(Z-\lambda_i)} \xrightarrow{\cong}~ \frac{R_{i-1}}{(Z- \lambda_i)}$, for every $i$, $1 \leqslant i \leqslant n$, the maps 
 $\mu_{j-1}$'s are isomorphisms for every $j \geqslant 1$.
	Hence from (\ref{cd3}), the following exact sequence is also split exact:
	$$
	\begin{tikzcd}
		0 \arrow[r] & G_{j}(R_{i-1})\arrow{r} & G_{j}\left(R_{i-1}\left[ \frac{1}{Z-\lambda_i}\right]\right) \arrow{r} & G_{j-1}\left(\frac{R_{i-1}}{Z-\lambda_i}\right) \arrow{r} & 0
	\end{tikzcd}
	$$
	 Thus the inclusion $R_{i-1} \hookrightarrow R_{i}=R_{i-1}[\frac{1}{Z-\lambda_i}]$ will induce a group isomorphism
	$$
	G_{j}(R_i) \cong G_{j}(R_{i-1}) \oplus G_{j-1}(R)
	$$ 
	for every $j \geqslant 1$ and $i$,  $1 \leqslant i \leqslant n$. In particular, for $j=1$, inductively we get that the inclusion 
	$\begin{tikzcd} R[Z] \arrow[r,hook,"\gamma_2"] &R_n=R\left[ Z, \frac{1}{a_{1}(Z)} \right]\end{tikzcd}$ induces the isomorphism 
	\begin{equation}\label{6}
		G_{1}\left(R\left[Z, \frac{1}{a_{1}(Z)} \right]\right) \cong G_{1}(R[Z]) \oplus \left(G_{0}(R) \right)^{n}.
	\end{equation}
	As $G_{0}(R) \neq 0$ by \thref{g0}, \eqref{6} contardicts \eqref{5}  if $n >0$ i.e., if $a_{1}(Z) \notin k^{*}$. 
	Therefore, we have $a_{1}(Z) \in k^{*}$. Hence we obtain that $k[Z,T]=k[f]^{[1]}$.
	
	\medskip
	
	When $k$ is not algebraically closed, consider the ring $\overline{A}= A \otimes_{k} \overline{k}$, where $\overline{k}$ is an algebraic closure of $k$. Note that $\overline{A}^{[l]}=\overline{k}^{[m+l+2]}$ and $\overline{k}[x_{1},\ldots,x_{m},z,t] \subsetneqq \dk(\overline{A})$. Hence by the above argument, we have $\overline{k}[Z,T]=\overline{k}[f]^{[1]}$.

	\noindent
	\medskip
	{\it Case }1:
	If $k$ is a finite field, then by \thref{sepco}, $k[Z,T]=k[f]^{[1]}$.

	\noindent
	\medskip
	{\it Case }2:
	Let $k$ be an infinite field. Then by \thref{dk} we can assume that $f(Z,T)=a_{0}(Z)+a_{1}(Z)T$ in $k[Z,T]$. As $f$ is a coordinate in $\overline{k}[Z,T]$, this is possible only if either $a_1(Z)=0$ and $a_0(Z)$ is linear in $Z$ or if $a_1(Z) \in \overline{k}^*$. Hence in either case 
	$$
	k[Z,T]=k[f]^{[1]}.
	$$
 
	\medskip

	\noindent
	$\rm (v) \Rightarrow (i):$
	Let $h \in k[Z,T]$ be such that $k[Z,T]=k[f,h]$. Therefore, without loss of generality,  we can assume that $f=Z$ and $h=T$. Hence $G=X_{1}^{r_{1}} \cdots X_{m}^{r_{m}}Y- X_{1} \cdots X_{m}g-Z$, for some $g \in k[X_{1},\ldots,X_{m},Z,T]$. Now by \thref{rsco}, we get that 
	$$
	k[X_{1},\ldots,X_{m},Y,Z,T]=k[X_{1},\ldots,X_{m},G,T]^{[1]}=k[X_{1},\ldots,X_{m},G]^{[2]}.
	$$
 Therefore the equivalence of the first six statements are established.
 
 \medskip

  By 
  \thref{r1}, $\rm (viii) \Rightarrow (v)$. By \thref{line}, $\rm (vii) \Leftrightarrow (viii)$ and $\rm (ix) \Leftrightarrow (xi)$. Therefore, $\rm (viii) \Rightarrow (v) \Leftrightarrow (iii) \Rightarrow (vii) \Leftrightarrow (viii)$. 
  
 \medskip
 
 By \thref{subdk} and \thref{ml}(b), $\rm (x) \Rightarrow (vi)$,  $\rm (xi) \Rightarrow (vii)$.
 We now see that the following hold:
 
 \smallskip
 \noindent
 $\rm (ix) \Leftrightarrow (xi) \Rightarrow (vii) \Leftrightarrow (iii) \Rightarrow (xi)$.
 
 \smallskip
 \noindent	
 $\rm (x) \Rightarrow (vi)\Leftrightarrow (iii) \Rightarrow (x)$.
	
Hence equivalence of all the statements are established.
\end{proof}

The following example shows that without the hypothesis that $F$ is of the form as in \thref{equiv}, the condition $k[Z,T]=k[f]^{[1]}$ is not sufficient for $A$ to be $k^{[m+2]}$, for $m \geqslant 2$. In particular, it is not sufficient to ensure $k[X_1,\dots, X_m, Y,Z,T] = k[X_1,\dots,X_m,G]^{[2]}$.

	\begin{ex}\thlabel{ex1}
	\em{Let  $$A = \frac{k[X_1,X_2,Y,Z,T]}{(X_1^2X_2^2Y-F)},$$
		 where $F(X_1,X_2,Z,T)=X_1Z+X_2+Z$. Note that here $f(Z,T)=F(0,0,Z,T)=Z$ is a coordinate of $k[Z,T]$. Also it is easy to see that $A$ is regular. But as 
		  $F(X_1,0,Z,T)=X_1Z+Z$ is neither irreducible nor a unit, by \thref{ufd2}, $A$ is not even a UFD. Therefore, $A \neq k^{[4]}$.}
\end{ex}

 We now prove two additional statements which are equivalent to each of the eleven statements in \thref{equiv}.
We start with two technical lemmas which will be used in \thref{eq13}. The aim of these lemmas is to show that if $\dk(A)=A$, then we can construct another affine domain $A_l$ of similar structure to $A$, such that $\dk(A_l)=A_l$ and $\dim(A_l)<\dim(A)$.

\begin{lem}\thlabel{fdk1}
	Suppose $k[x_{1}, \ldots , x_{m},z,t] \subsetneqq \dk(A)$. Then there exists a graded integral domain 
	$$
	\widehat{A} \cong \frac{k[X_{1}, \ldots , X_{m}, Y,Z,T]}{(X_{1}^{r_{1}} \cdots X_{m}^{r_{m}}Y-f( Z,T))}
	$$ 
	and a non-trivial, homogeneous exponential map $\widehat{\phi}$ 
on $\widehat{A}$, such that $wt(\widehat{x_{i}})=-1$ for every $i \in \{1,\ldots,m\}$, $wt(\widehat{z})=wt(\widehat{t})=0$, $wt(\widehat{y})=r_1+\cdots+r_m$ and $\widehat{y} \in \widehat{A}^{\widehat{\phi}}$, 
where $\widehat{x_{1}},\ldots, \widehat{x_{m}},\widehat{y},\widehat{z}, \widehat{t}$ 
denote the images of $X_1,\ldots,X_m,Y,Z,T$ respectively in $\widehat{A}$.
\end{lem}
\begin{proof}
	Since $ k[x_{1}, \ldots, x_{m}, z,t] \subsetneqq \dk(A) $, there exist an exponential map $\phi$ on $A$ 
	and $h \in A^{\phi}$ such that $h$ has a monomial summand of the form 
	$h_{\overline{\iota}}=x_{1}^{i_{1}} \cdots x_{m}^{i_{m}}y^{j}z^{p}t^{q}$,
	where $\overline{\iota}=(i_{1},\ldots,i_{m}) \in \mathbb{Z}_{\geqslant0}^m$, $j \geqslant 1, p \geqslant 0, q \geqslant 0$ and there exists some $i_{s}$ such that $i_{s}< r_{s}$. Without loss of generality, we assume that $i_{s}=i_{1}$. 
	
	We now consider the proper $\mathbb{Z}$-filtration on $A$ with respect to $(-1,0,\ldots,0) \in \mathbb{Z}^{m}$.
	Let $\widetilde{A}$ be the associated graded ring and $\widetilde{F}(x_{1}, \ldots , x_{m},z,t)$ 
denote the highest degree 
homogeneous summand of $F(x_{1}, \ldots , x_{m},z,t)$. Then 
$\widetilde{F}(x_{1}, \ldots , x_{m},z,t)= F(0,x_{2}, \ldots , x_{m},z,t)$. Since $f(z,t) \neq 0$, we have $x_{i} \nmid \widetilde{F}$, for every $1 \leqslant i \leqslant m$. Therefore, by \thref{gr}, we have 
		$$
	\widetilde{A} \cong \frac{k[X_{1}, \ldots , X_{m}, Y,Z,T]}{(X_{1}^{r_{1}} \cdots X_{m}^{r_{m}}Y-\widetilde{F})}.
	$$
	For every $\alpha \in A$, let $\widetilde{\alpha}$ 
	denote its image in $\widetilde{A}$.
	By \thref{dhm}, $\phi$ will induce a non-trivial homogeneous exponential map $\widetilde{\phi}$ on $\widetilde{A}$ such that $\widetilde{h} \in \widetilde{A}^{\widetilde{\phi}}$. From the chosen filtration on $A$, it is clear that $\widetilde{y}\,  |\, \widetilde{h}$ and hence $\widetilde{y} \in \widetilde{A}^{\widetilde{\phi}}$. 
	
	We now consider the $\mathbb{Z}$-filtration on $\widetilde{A}$ with respect to $(-1,\ldots,-1) \in \mathbb{Z}^m$.
	By \thref{dhm}, we have a homogeneous non-trivial exponential map $\widehat{\phi}$ on the associated graded ring $\widehat{A}$ induced by $\widetilde{\phi}$. By \thref{gr}, 
	$$
	\widehat{A} \cong \frac{k[X_{1}, \ldots , X_{m}, Y,Z,T]}{(X_{1}^{r_{1}} \cdots X_{m}^{r_{m}}Y-f(Z,T))},
	$$
	since $f(Z,T)=\widetilde{F}(0, \ldots , 0, Z,T)$ is the highest degree homogeneous summand of $\widetilde{F}$.
 For every $a \in \widetilde{A}$, let $\widehat{a}$ denote its image in $\widehat{A}$. Since $\widetilde{y} \in \widetilde{A}^{\widetilde{\phi}},$ we have $\widehat{y} \in \widehat{A}^{\widehat{\phi}}$. The weights of $\widehat{x_{1}},\ldots,\widehat{x_{m}},\widehat{y},\widehat{z},\widehat{t}$ are clear from the chosen filtration on $\widetilde{A}$.
\end{proof}

\begin{lem}\thlabel{fdk2}
	Suppose $m>1$ and $k[x_{1}, \ldots , x_{m},z,t] \subsetneqq \dk(A)$. Then there exists an integer $l \in \{1,\ldots,m\}$ and an integral domain $A_l$ such that 
	$$
	A_{l} \cong \frac{k(X_{l})[X_{1}, \ldots , X_{l-1}, X_{l+1}, \ldots, X_{m}, Y,Z,T]}{(X_{1}^{r_{1}} \cdots X_{m}^{r_{m}}Y-f( Z,T))},
	$$ 
	where the image  $y^{\prime}$ of $Y$ in $A_l$ belongs to $\dk(A_l)$. In particular, $\dk(A_{l})= A_{l}$.
\end{lem}

\begin{proof}
	
	By \thref{fdk1}, there exists an integral domain 
	$$
	\widehat{A} \cong \frac{k[X_{1}, \ldots , X_{m}, Y,Z,T]}{(X_{1}^{r_{1}} \cdots X_{m}^{r_{m}}Y-f(Z,T))},
	$$
	and a non-trivial homogeneous exponential map $\widehat{\phi}$ on $\widehat{A}$ such that for every $i \in \{1,\ldots,m\}$, $wt(\widehat{x_{i}})=-1$, $wt(\widehat{z})=wt(\widehat{t})=0$, $wt(\widehat{y})=r_1+\cdots+r_m$ and $\widehat{y} \in \widehat{A}^{\widehat{\phi}}$, 
where $\widehat{x_{1}},\ldots, \widehat{x_{m}},\widehat{y},\widehat{z}, \widehat{t}$ 
denote the images of $X_1,\ldots,X_m,Y,Z,T$ respectively in $\widehat{A}$.
	
	Since $m >1$, $\td_{k}(\widehat{A}^{\widehat{\phi}} ) \geqslant 3$. Hence if $\widehat{A}^{\widehat{\phi}} \subseteq k[\widehat{y},\widehat{z},\widehat{t}]$, then $\widehat{A}^{\widehat{\phi}} =k[\widehat{y},\widehat{z},\widehat{t}]$.  But then $\widehat{x}_{1}^{r_{1}} \cdots \widehat{x}_{m}^{r_{m}}\widehat{y}= f(\widehat{z},\widehat{t}) \in \widehat{A}^{\widehat{\phi}} $, that means $\widehat{\phi}$ is trivial, as $\widehat{A}^{\widehat{\phi}}$ is factorially closed (cf. \thref{prope}(i)). This is a contradiction. 
	Therefore, there exists $h_{1} \in \widehat{A}^{\widehat{\phi}} \setminus k[\widehat{y}, \widehat{z}, \widehat{t}]$, which is homogeneous with respect to the grading on $\widehat{A}$. Let
	
	$$
	h_{1}= h^{\prime}(\widehat{x_{1}},\ldots,\widehat{x_{m}},\widehat{z},\widehat{t})+ \sum_{\overline{\iota}=(i_{1},\ldots,i_{m}) \in \mathbb{Z}_{\geqslant 0}^m,\, j>0\, p,q \geqslant 0} \lambda_{\overline{\iota}\, jpq}\, \widehat{x_{1}}^{i_{1}} \cdots \widehat{x_{m}}^{i_{m}} \widehat{y}^j \widehat{z}^p \widehat{t}^q
	$$
	such that $\lambda_{\overline{\iota}\, jpq} \in k$ and for every $j>0$ and $\overline{\iota}=(i_{1},\ldots,i_{m}) \in \mathbb{Z}_{\geqslant 0}^m$, there exists $s_{j} \in \{1,\ldots,m\}$ such that $i_{s_{j}} < r_{s_{j}}$. Now we have the following two cases:
	
	\smallskip
	\noindent
	{\it Case  }1:
	If $h^{\prime} \notin k[\widehat{z},\widehat{t}]$, then it has a monomial summand $h_2$ such that $\widehat{x_{l}}\, | \, h_2$, for some $l, 1 \leqslant l \leqslant m$.
	
	\smallskip
	\noindent
	{\it Case }2:
	If $h^{\prime} \in k[\widehat{z},\widehat{t}]$, then each of the monomial summands of the form $\lambda_{\overline{\iota}\, jpq} \widehat{x_{1}}^{i_{1}} \cdots \widehat{x_{m}}^{i_{m}} \widehat{y}^j \widehat{z}^p \widehat{t}^q$ of $h_{1}$ has degree zero. That means, $j(r_{1}+\cdots+r_{m})=i_{1}+\cdots+i_{m}$. Therefore, for every $j>0$ and $\overline{\iota} \in \mathbb{Z}_{\geqslant 0}^m$, there exists $l_{j} \in \{1,\ldots,m\}, l_{j} \neq s_{j}$ such that $i_{l_{j}}> jr_{l_{j}}$, as $i_{s_{j}}<r_{s_{j}}$. We choose one of these $l_{j}$'s and call it $l$.  
	
	We consider the $\mathbb{Z}$-filtration on $\widehat{A}$ with respect to $(0,\ldots,0,1,0,\ldots,0) \in \mathbb{Z}^m$, where the $l$-th entry is 1. Let $\overline{A}$ be the associated graded ring of $\widehat{A}$, and by \thref{gr},
	$$
	\overline{A} \cong \frac{k[X_{1}, \ldots , X_{m}, Y,Z,T]}{(X_{1}^{r_{1}} \cdots X_{m}^{r_{m}}Y-f(Z,T))}.
	$$
	For every $a \in \widehat{A}$, let $\overline{a}$ denote its image in $\overline{A}$. By \thref{dhm},
	$\widehat{\phi}$ induces a non-trivial exponential map $\overline{\phi}$ on the associated graded ring $\overline{A}$ such that $\overline{h_{1}} \in \overline{A}^{\overline{\phi}}$. From {\it Case }1 and {\it Case }2, it is clear that $\overline{h_{1}}$ is divisible by $\overline{x_{l}}$. Hence we get that $\overline{x_{l}} \in \overline{A}^{\overline{\phi}}$.  
	Therefore, by \thref{prope}(iii), $\overline{\phi}$ will induce a non-trivial exponential map $\phi_{l}$ on $A_{l}= \overline{A} \otimes_{k[\overline{x_{l}}]} k(\overline{x_{l}})$, where
	$$
	A_{l} \cong \frac{k(X_{l})[X_{1}, \ldots , X_{l-1}, X_{l+1}, \ldots , X_{m}, Y,Z,T]}{(X_{1}^{r_{1}} \cdots X_{m}^{r_{m}}Y-f(Z,T))}.
	$$ 
	Since $A_{l}^{\phi_{l}}= \overline{A}^{\overline{\phi}} \otimes_{k[\overline{x_{l}}]} k(\overline{x_{l}})$ 
and $\overline{y} \in \overline{A}^{\overline{\phi}}$, the image $y^{\prime}$ of $Y$ in $A_{l}$ is in $A_{l}^{\phi_{l}}$. Therefore, by \thref{subdk}, we get that $DK(A_{l})=A_{l}$.
\end{proof}

We now add two more equivalent statements to \thref{equiv}.
 \begin{thm}\thlabel{eq13}
	Let $A$ be the affine domain as in \eqref{AG} and $	F(X_{1}, \ldots, X_{m}, Z,T)= f(Z,T)+ (X_{1} \cdots X_{m})g$, for some $g \in k[X_1,\ldots,X_m,Z,T]$. Then the following statements are equivalent. 
	\begin{enumerate}
		\item[\rm(i)] $A$ is a UFD, $k[x_{1},\ldots,x_{m},z,t] \subsetneqq \dk(A)$ and  $\left( \frac{A}{x_{i}A}  \right)^{*}=k^{*}$, for every $i\in \{1,\ldots,m\}$.
		
		\item[\rm(ii)] $k[Z,T]=k[f]^{[1]}$.
		
		\item[\rm(iii)] $A=k[x_1,\ldots,x_m]^{[2]}$.

		\item[\rm(iv)] $A$ is a UFD, $\ml(A)=k$ and  $\left( \frac{A}{x_{i}A}  \right)^{*}=k^{*}$ for every $i \in \{1,\ldots,m\}$.
	\end{enumerate} 
\end{thm}
\begin{proof}
	$\rm (ii) \Leftrightarrow (iii)$ follows from \thref{equiv}. $\rm (iii) \Rightarrow (iv)$ holds trivially and
	$\rm (iv) \Rightarrow (i)$  follows by \thref{ml}(b). 
	
	\medskip
	\noindent
	$\rm (i) \Rightarrow (ii):$  We prove this by induction on $m$. We consider the case $m=1$. Since $k[x_1,z,t] \subsetneqq \dk(A)$, by \thref{r2}, without loss of generality we can assume that $f(Z,T)=a_0(Z)+a_1(Z)T$ for some $a_0,a_1 \in k^{[1]}$. Since $A$ is a UFD, either $f(Z,T)$ is irreducible or $f(Z,T) \in k^{*}$ (cf. \thref{ufd2}). If $f(Z,T) \in k^{*}$, then $x_1 \in A^{*}$, which contradicts that $\left(\frac{A}{x_1A}\right)^{*}=k^{*}$. Therefore, $f(Z,T)$ is irreducible in $k[Z,T]$. Note that $\frac{A}{x_1A} \cong  \frac{k[Y,Z,T]}{(f(Z,T))}$.
	If $a_1(Z)=0$, then $f(Z,T)=a_0(Z)$. Since $\left(\frac{A}{x_1A}\right)^{*}=\left( \frac{k[Y,Z,T]}{(a_0(Z))}\right)^{*}=k^{*}$, $a_0(Z)$ must be linear in $Z$. Hence $k[Z,T]=k[f]^{[1]}$. 
	If $a_1(Z) \neq 0$, then ${\rm gcd}(a_0,a_1)=1$, as $f(Z,T)$ is irreducible. Therefore, since $k^{*}=\left(\frac{A}{x_1A}\right)^{*}=\left( \frac{k[Y,Z,T]}{(f(Z,T))}\right)^{*}=\left(k\left[Y,Z,\frac{1}{a_1(Z)} \right] \right)^{*}$, $a_1(Z) \in k^{*}$. Thus, $k[Z,T]=k[f]^{[1]}$.
	
	We now assume $m>1$ and the result holds upto $m-1$. 
	Since $k[x_1,\ldots,x_m,z,t] \subsetneqq \dk(A)$, by \thref{fdk2}, there exists an integral domain 
	$$
	A_{l} \cong \frac{k(X_{l})[X_{1}, \ldots , X_{l-1}, X_{l+1}, \ldots , X_{m}, Y,Z,T]}{(X_{1}^{r_{1}} \cdots X_{m}^{r_{m}}Y-f(Z,T))},
	$$ 
	such that $\dk(A_l)=A_l$. Note that, for every $i \in \{1,\ldots,m\}$, 
$$
\frac{A}{x_iA} \cong \frac{k[X_1,\ldots,X_{i-1},X_{i+1},\ldots,X_{m},Y,Z,T]}{(f(Z,T))}.
$$
Now for every $i \neq l$, 
	$$
	\frac{A}{x_{i}A} \otimes_{k[x_l]} k(x_l) \cong \frac{k(X_l)[X_1,\ldots,X_{i-1},X_{i+1},\ldots,X_{l-1},X_{l+1},\ldots,X_m,Y,Z,T]}{(f(Z,T))} \cong \frac{A_l}{\overline{x_i}A_l},
	$$
	where $\overline{x_i}$ denotes the image of $X_i$ in $A_l$.
	Since $\left( \frac{A}{x_iA} \right)^{*}=k^{*}$, it follows that 
	$$
	\left( \frac{A_l}{\overline{x_i}A_l} \right)^{*} \cong k(X_l)^{*}.
	$$ Since $A$ is a UFD and $x_i$'s are not unit in $A$, by \thref{ufd2}, $f(Z,T)$ is irreducible in $k[Z,T]$. As $\dk(A_l)=A_l$, by \thref{dk}, there exist $a_0,a_1 \in k(X_l)^{[1]}$ such that  $$
	f(Z,T)=a_0(Z_1)+a_1(Z_1)T_1,
	$$ for some $Z_1,T_1 \in k(X_l)[Z,T]$ such that $k(X_l)[Z,T]=k(X_l)[Z_1,T_1]$.
	
	We fix some $i, i \neq l$. Suppose $a_1(Z_1)=0$. Since 
	$\left(\frac{A_l}{\overline{x_{i}}A_l}\right)^{*} \cong \left( \frac{k(X_l)[Z_1,T_1]}{(a_0(Z_1))} \right)^{*} \cong k(X_l)^{*}$, and $a_0(Z_1)$ is irreducible in $k(X_l)[Z_1]$, it follows that $a_0(Z_1)$ is linear in $Z_1$. Therefore, $f(Z,T)=a_{0}(Z_1)$ is a coordinate in $k(X_l)[Z,T]$.
	
	Suppose $a_1(Z_1) \neq 0 $. Since $f(Z,T)$ is irreducible in $k(X_l)[Z_1,T_1]$, it follows that ${\rm gcd}(a_0(Z_1),a_1(Z_1))=1$ and hence 
	$$
	\frac{A_l}{\overline{x_{i}}A_l} \cong \frac{k(X_l)[Z_1,T_1]}{(a_0(Z_1)+a_1(Z_1)T_1)} \cong k(X_l)\left[ Z_1, \frac{1}{a_1(Z_1)} \right].
	$$
	Therefore, $a_1(Z_1) \in \left(\frac{A_l}{\overline{x_i}A_l}\right)^{*} \cong k(X_l)^{*}$.  
	
	Hence we have $k(X_l)[Z,T]=k(X_l)[f]^{[1]}$. Therefore, by \thref{sepco}, we have $k[Z,T]=k[f]^{[1]}$. 
\end{proof}

\begin{rem}
	\em{ (i) The above result shows that the condition ``$A$ is geometrically factorial" (i.e., $A\otimes_{k} \overline{k}$ is a UFD, where $\overline{k}$ is an algebraic closure of $k$) in \cite[Theorem 3.11(viii)]{com} can be relaxed to ``$A$ is a UFD".

	(ii) In \cite[Theorem 4.5]{GG1}, it has been shown that if $A$ is as in Theorem \ref{eq13}, then the condition $A= k^{[m+2]}$ is equivalent to another condition:

	``$A$ is geometrically factorial over $k$ and there exists an exponential map $\phi$ on $A$ satisfying $k[x_1,\ldots, x_m] \subseteq A^{\phi} \not \subseteq k[x_1,\ldots, x_m,z,t]$.''}
\end{rem}

Below we show that Theorem B in section 1 has a generalisation over a larger class of integral domains. 
The $m=1$ case of the following theorem has been proved in \cite[Theorem 4.9]{DG}. 

\begin{thm}\thlabel{geqv}
	Let $R$ be a Noetherian integral domain such that either $\mathbb{Q}$ is contained in $R$ or $R$ is seminormal. Let
$$
A_R:= \frac{R\left[ X_{1}, \ldots , X_{m}, Y,Z,T\right]}{\left(X_{1}^{r_{1}} \cdots X_{m}^{r_{m}}Y- F(X_{1}, \ldots, X_{m}, Z,T) \right)},
~~r_{i} >1  \text{~for all~} i, 1 \leqslant i \leqslant m,
$$  
where $F(X_{1}, \ldots , X_{m},Z,T)=f(Z,T)+ (X_{1} \cdots X_{m})g(X_{1}, \ldots , X_{m},Z,T)$ 
and $f(Z,T) \neq 0$. Let $G=X_{1}^{r_{1}} \cdots X_{m}^{r_{m}}Y- F(X_{1}, \ldots, X_{m}, Z,T)$ 
and $\widetilde{x_{1}}, \ldots ,\widetilde{ x_{m}}$ denote the images in $A_R$ of $X_{1}, \ldots , X_{m}$
 respectively.  Then the following statements are equivalent:
\begin{enumerate}
\item [\rm(i)] $R[X_{1}, \ldots , X_{m},Y,Z,T]=R[X_{1}, \ldots, X_{m},G]^{[2]}$.

\item[\rm(ii)]  $R[X_{1}, \ldots , X_{m},Y,Z,T]=R[G]^{[m+2]}$.

\item[\rm(iii)] $A_R=R[\widetilde{x_{1}}, \ldots , \widetilde{x_{m}}]^{[2]}$.

\item[\rm(iv)] $A_R=R^{[m+2]}$.

\item[\rm(v)] $R[Z,T]=R[f(Z,T)]^{[1]}$.

\end{enumerate}
\end{thm}

\begin{proof}
	Note that $\rm (i) \Rightarrow (ii) \Rightarrow (iv)$ and 
$\rm (i) \Rightarrow (iii) \Rightarrow (iv)$ follow trivially. 
Therefore it is enough to show $\rm (iv) \Rightarrow (v)$ and $\rm (v) \Rightarrow (i)$.
	
	\medskip
	
	\noindent
	$\rm (iv) \Rightarrow (v):$
Let $\p \in \Spec R$ and $\kappa(\p)= \frac{R_{\p}}{\p R_{\p}}$. Then $A \otimes_{R} \kappa(\p) =\kappa(\p)^{[m+2]}$. 
Now from $\rm (iv) \Rightarrow (v)$ of \thref{equiv}, we have $f$ is a residual coordinate in $R[Z,T]$. Hence $R[Z,T]=R[f]^{[1]}$ by \thref{bd}.

	\medskip
	
    \noindent
	$\rm (v) \Rightarrow (i):$
	Let $h \in R[Z,T]$ be such that $R[Z,T]=R[f,h]$. Therefore, without loss of generality we can assume that $f=Z$ and $h=T$. Hence $G=X_{1}^{r_{1}} \cdots X_{m}^{r_{m}}Y- X_{1} \cdots X_{m}g-Z$, for some $g \in R[X_{1},\ldots,X_{m},Z,T]$. Now by \thref{rsco}, we get that 
	$$
	R[X_{1},\ldots,X_{m},Y,Z,T]=R[X_{1},\ldots,X_{m},G,T]^{[1]}=R[X_{1},\ldots,X_{m},G]^{[2]}.
	$$ 
\end{proof}

\section{Isomorphism classes and Automorphisms}

In this section we will describe the isomorphism classes and automorphisms of integral domains of type $A$ (as in \eqref{AG}) when $\dk(A)=k[x_1,\ldots,x_m,z,t]$.

The following result describes some necessary conditions for two such rings to be isomorphic.

\begin{thm}\thlabel{isocl}
	Let $(r_1,\ldots,r_m),(s_1,\ldots,s_m) \in \mathbb{Z}^{m}_{>1}$,
	and $F, G \in k[X_1, \dots, X_m, Z,T]$,
	where $f(Z,T):= F(0,\dots,0,Z,T)\neq 0$ and $g(Z,T):= G(0,\dots,0,Z,T)\neq 0$. 
	Suppose $\phi: A \rightarrow A^{\prime}$ is an isomorphism, where
	$$
	A=A(r_1,\ldots,r_m,F):=
	\frac{k[X_1,\ldots,X_m,Y,Z,T]}{\left(X_1^{r_1}\cdots X_m^{r_m}Y-F(X_{1}, \ldots , X_{m},Z,T)\right)}
	$$
	and 
	$$
	A^{\prime}=A(s_1,\ldots,s_m,G):=
	\frac{k[X_1,\ldots,X_m,Y,Z,T]}{\left(X_1^{s_1}\cdots X_m^{s_m}Y-G(X_{1}, \ldots , X_{m},Z,T)\right)}.
	$$
	Let $x_1,\ldots,x_m,y,z,t$ and $x_1^{\prime},\ldots,x_m^{\prime},y^{\prime},z^{\prime},t^{\prime}$ 
	denote the images of $X_1,\ldots,X_m,Y,Z,T$ in $A$ and $A^{\prime}$ respectively. 
	Let $E=k[x_1,\ldots,x_m]$, $E^{\prime}=k[x_1^{\prime},\ldots,x_m^{\prime}]$, 
	$B=k[x_1,\ldots,x_m,z,t]$ and $B^{\prime}=k[x_1^{\prime},\ldots,x_m^{\prime},z^{\prime},t^{\prime}]$. Suppose 
	$ \dk(A)=B$ and  $\dk(A^{\prime})=B^{\prime}$.
	Then 
	\begin{enumerate}
		\item[\rm (i)] $\phi$ restricts to isomorphisms from $B$ to $B^{\prime}$ and from $E$ to $E^{\prime}$.
	
		\item[\rm (ii)] For each $i, 1 \leqslant i \leqslant m$, there exists $j$, $1 \leqslant j \leqslant m$, such that $\phi(x_{i})=\lambda_{j}x_{j}^{\prime}$ for some $\lambda_{j} \in k^{*}$ and $r_{i}=s_{j}$. In particular,
		$(r_1,\ldots,r_m)=(s_1,\ldots,s_m)$ upto a permutation of $\{1,\ldots,m\}$.
		
		\item[\rm (iii)] $\phi \left(( x_{1}^{r_{1}}\cdots x_{m}^{r_{m}}, F(x_{1}, \ldots, x_{m},z,t))B \right)
		=\left( (x_{1}^{\prime})^{s_{1}}\cdots (x_{m}^{\prime})^{s_{m}}, G(x_{1}^{\prime}, \ldots, x_{m}^{\prime},z^{\prime},t^{\prime}) \right)B^{\prime}$.
		
		\item[\rm(iv)] There exists $\alpha \in \Aut_k(k[Z,T])$ such that $\alpha(g)=\lambda f$ for some $\lambda \in k^{*}$.
	\end{enumerate}
 \end{thm}

\begin{proof}
	(i)	Since $\phi: A \rightarrow A^{\prime}$ is an isomorphism, $\phi$ restricts to an isomorphism of the Derksen invariant and the Makar-Limanov invariant. Therefore, 
	$\phi(B)=B^{\prime}$.  By \thref{ml}(b), $\ml(A)=E$ and $\ml(A^{\prime})=E^{\prime}$.
	Hence $\phi(E)=E^{\prime}$.

{\it We now identify $\phi(A)$ with $A$ and assume that $A^{\prime}=A$, $\phi$ is identity on $A$, $B^{\prime}=B$ and $E^{\prime}=E$.}
 
	\medskip
	
	\noindent
	(ii) 
	We first show that for every $i, 1 \leqslant i \leqslant m$, $x_i= \lambda_j x_j^{\prime}$, 
for some $j, 1\leqslant j \leqslant m$ and $\lambda_{j} \in k^*$.
We now have 
	$$
	y^{\prime} =\frac{G(x_1^{\prime},\ldots,x_m^{\prime},z^{\prime},t^{\prime})}
	{(x_1^{\prime})^{s_1}\cdots(x_m^{\prime})^{s_m} } \in A \setminus B.
	$$ 
	Since $A \hookrightarrow k[x_1^{\pm 1},\ldots,x_m^{\pm 1},z,t]$, there exists $n>0$ such that 
	$$
	(x_1\cdots x_m)^n y^{\prime} = \frac{(x_1\cdots x_m)^n G(x_1^{\prime},\ldots,x_m^{\prime},z^{\prime},t^{\prime})}
	{(x_1^{\prime})^{s_1}\cdots(x_m^{\prime})^{s_m}} \in B.
	$$
	Since for every $j \in \{1,\ldots,m\}$, $x_j^{\prime}$ is irreducible in $B$, and $x_j^{\prime}\nmid G$ in $B$,
	we have $x_j^{\prime}\mid (x_1\cdots x_m)^n$. Since $x_1, x_2, \dots, x_m$ are also irreducibles in $B$, we have
	\begin{equation}\label{x}
		x_i=\lambda_jx_j^{\prime},
	\end{equation}
	for some $i\in \{1,\ldots,m\}$ and $\lambda_{j} \in k^{*}$. 
	
	We now show that $r_i=s_j$. 
	Suppose $r_i>s_j$. 
	Consider the ideal 
	$$
	\mathfrak{a}_i := x_i^{r_i}A\cap B = \left( x_i^{r_i}, F(x_1,\ldots,x_m, z, t) \right)B.
	$$
	Again by \eqref{x}, 
	$$
	\mathfrak{a}_i= (x_j^{\prime} )^{r_i}A \cap B 
	= (x_j^{\prime} )^{r_i}A^{\prime} \cap B^{\prime}=
	 \left( (x_j^{\prime} )^{r_i}, (x_j^{\prime} )^{r_i-s_j}G(x_1^{\prime} ,\ldots,x_m^{\prime} ,z^{\prime} ,t^{\prime} ) \right)B^{\prime} \subseteq x_j^{\prime}B^{\prime},
	$$ 
	which implies that $F(x_1,\ldots,x_m,z,t) \in x_j^{\prime}  B^{\prime} = x_j^{\prime}B$. 
	But this is a contradiction. 
	Therefore, $r_i \leqslant s_j$. 
	By similar arguments, we have $s_j \leqslant r_i$. Hence $r_i=s_j$ and as $i\in \{1,\ldots,m\}$ is arbitrary, the assertion follows.
	
	\medskip
	\noindent
	(iii) We now show that
	$$
	\left( x_{1}^{r_{1}}\cdots x_{m}^{r_{m}}, F(x_{1}, \ldots, x_{m},z,t) \right)B
	=\left( (x_{1}^{\prime})^{s_{1}}\cdots (x_{m}^{\prime})^{s_{m}}, G(x_{1}^{\prime}, \ldots, x_{m}^{\prime},z^{\prime},t^{\prime}) \right)B.
	$$
	From (ii), it is clear that $x_{1}^{r_{1}}\cdots x_{m}^{r_{m}}= \mu (x_{1}^{\prime})^{s_{1}}\cdots (x_{m}^{\prime})^{s_{m}} $, for some $\mu \in k^*$. Since 
	$$
	(x_{1}^{r_{1}}\cdots x_{m}^{r_{m}})A \cap B= \left( x_{1}^{r_{1}}\cdots x_{m}^{r_{m}}, F(x_{1}, \ldots, x_{m},z,t) \right)B
	$$
	and 
	$$
	\left( (x_{1}^{\prime})^{s_{1}}\cdots (x_{m}^{\prime})^{s_{m}}\right)A \cap B=
	\left( (x_{1}^{\prime})^{s_{1}}\cdots (x_{m}^{\prime})^{s_{m}}, G(x_{1}^{\prime}, \ldots, x_{m}^{\prime},z^{\prime},t^{\prime}) \right)B,
	$$
	the result follows.
	
	\medskip
	\noindent
	(iv)  Since $\left( x_{1}^{r_{1}}\cdots x_{m}^{r_{m}}, F(x_{1}, \ldots, x_{m},z,t) \right)B
	=\left( (x_{1}^{\prime})^{s_{1}}\cdots (x_{m}^{\prime})^{s_{m}}, G(x_{1}^{\prime}, \ldots, x_{m}^{\prime},z^{\prime},t^{\prime}) \right)B$, 
from \eqref{x}, it follows that 
	\begin{equation}\label{g}
		g(z^{\prime},t^{\prime})= \lambda f(z,t)+ H(x_1,\ldots,x_m,z,t),
	\end{equation}
	for some $\lambda \in k^{*}$ and $H \in \left( x_1,\ldots,x_m\right)B$. 
Let $z^{\prime}=h_1(x_1,\ldots,x_m,z,t)$ and $t^{\prime}=h_2(x_1,\ldots,x_m,z,t)$. Then we have $k[z,t]=k[h_1(0,\ldots,0,z,t),h_2(0,\ldots,0,z,t)]$. Hence $\alpha: k[Z,T] \rightarrow k[Z,T]$ defined by $\alpha(Z)=h_1(0,\ldots,0,Z,T)$ and $\alpha(T)=h_2(0,\ldots,0,Z,T)$ gives an automorphism of $k[Z,T]$, and from \eqref{g}, it follows that $\alpha(g)=\lambda f$. 
\end{proof}

The next result characterises the automorphisms of $A$ when $\dk(A)= B=k[x_{1}, \ldots , x_{m},z,t]$.

\begin{thm}\thlabel{aut1}
	Let $A$ be the affine domain as in \eqref{AG}, where 
	$$
	F(X_{1}, \ldots , X_{m},Z,T)=f(Z,T)+ h(X_{1},\ldots,X_{m},Z,T),
	$$
	for some $h \in (X_1,\ldots,X_m) k[X_1,\ldots,X_m,Z,T]$. As before, $x_1,\ldots,x_m,y,z,t$ denote the images of $X_1,\ldots,X_m,Y,Z,T$ in $A$.
	Suppose $\dk(A)=B=k[x_1,\ldots,x_m,z,t]$. If $\phi \in \Aut_{k}(A)$, then the following hold:	
	\begin{enumerate}
		\item [\rm(a)]  $\phi$ restricts to an automorphism of $E=k[x_{1},\ldots,x_{m}]$ and $B=k[x_{1}, \ldots , x_{m},z,t]$.
	\item[\rm(b)] For each $i$, $1 \leqslant i \leqslant m$, there exists $j$, $1 \leqslant j \leqslant m$ such that $\phi(x_{i})=\lambda_{j}x_{j}$, where $\lambda_{j} \in k^{*}\,\,(1 \leqslant i,j \leqslant m)$ and $r_{i}=r_{j}$.
		
		\item[\rm(c)] $\phi(I)=I$, where $I=\left( x_{1}^{r_{1}}\cdots x_{m}^{r_{m}}, F(x_{1}, \ldots, x_{m},z,t) \right)k[x_{1}, \ldots , x_{m},z,t]$.  
		\end{enumerate}
		Conversely, if $\phi \in \End_k(A)$ satisfies the conditions \rm(a) and \rm(c), then (b) holds and $\phi \in \Aut_k(A)$.
\end{thm}

\begin{proof}
	 By \thref{ml}(b), $\ml(A)=E=k[x_1,\ldots,x_m]$. Now the statements (a), (b), (c) follow from \thref{isocl}(i), (ii), (iii) respectively.
	
	\medskip
	
	We now show the converse part. From (a) and (c),
	 it follows that $\phi(B)=B$, $\phi(E)=E$ and $\phi(I)=I$. 
	 Hence $\phi(I \cap E)=I \cap E=\left(x_{1}^{r_1}\cdots x_{m}^{r_{m}} \right)E$, and therefore,
	$$
	\phi(x_{1}^{r_1}\cdots x_{m}^{r_{m}}) = \lambda x_{1}^{r_1}\cdots x_{m}^{r_{m}},
	$$
	for some $\lambda \in k^*$. Fix $i \in \{1,\ldots,m\}$. 
Since $x_i$ and $\phi(x_{i})$ are irreducibles in $E$, $\phi(x_{i})=\lambda_{j}x_{j}$, for some $\lambda_{j} \in k^{*}$ and $j \in \{1,\ldots,m\}$ and hence $r_{i} = r_{j}$ and (b) follows.

	Since $\phi$ is an automorphism of $B$ and $A \subseteq B[(x_1\cdots x_m)^{-1}]$, $\phi$ is an injective endomorphism of 
$A$, by (b). Therefore, it is enough to show that $\phi$ is surjective. For this, it is enough to find a preimage of $y$ in $A$. Since $\phi(I)=I$, we have $F= x_{1}^{r_{1}} \cdots x_{m}^{r_{m}} u(x_{1}, \ldots, x_{m},z,t) + \phi(F) v(x_{1}, \ldots, x_{m},z,t) $, for some $u,v \in B$. Since
	$
	y= \frac{F(x_{1}, \ldots , x_{m},z,t)}{x_{1}^{r_{1}} \cdots x_{m}^{r_{m}}}
	$, using (b), 
	\begin{equation}\label{y}
		y=u(x_{1},\ldots ,x_{m},z,t) + \frac{\phi(F) v(x_{1}, \ldots, x_{m},z,t)}{\lambda^{-1} \phi(x_{1}^{r_{1}} \cdots x_{m}^{r_{m}})}.
	\end{equation}
	Since $\phi(B)=B$, there exist $\widetilde{u},\widetilde{v} \in B$ such that $\phi(\widetilde{u})=u$ and $\phi(\widetilde{v})=v$. And hence from \eqref{y}, we get that 
	$y= \phi(\widetilde{u}+\lambda y\widetilde{v})$ where $\widetilde{u}+\lambda y\widetilde{v} \in A$.
\end{proof} 

Let $(r_1,\ldots,r_m) \in \mathbb{Z}^{m}_{> 1}$ and $f(Z,T)$ be a non-trivial line in $k[Z,T]$. Let 
$$
A(r_1,\ldots,r_m,f):=\frac{k[X_1,\ldots,X_m,Y,Z,T]}{\left(X_1^{r_1}\cdots X_m^{r_m}Y-f(Z,T)\right)}.
$$
Our next result determines the isomorphism classes among the family of rings defined above.

\begin{thm}\thlabel{isocl1}
	Let $(r_1,\ldots,r_m),(s_1,\ldots,s_m) \in \mathbb{Z}^{m}_{>1}$, and $f,g \in k[Z,T]$ be non-trivial lines. Then $A(r_1,\ldots,r_m,f) \cong A(s_1,\ldots,s_m,g)$ if and only if $(r_1,\ldots,r_m)=(s_1,\ldots,s_m)$ upto a permutation of $\{1,\ldots,m\}$ and there exists $\alpha \in \Aut_k(k[Z,T])$ such that $\alpha(g)=\mu f$, for some $\mu \in k^{*}$.
\end{thm}

\begin{proof}
Suppose $A(r_1,\ldots,r_m,f) \cong A(s_1,\ldots,s_m,g)$ and let  
$x_1,\ldots,x_m,y,z,t$ and\\ $x_1^{\prime},\ldots,x_m^{\prime},y^{\prime},z^{\prime},t^{\prime}$
 denote the images of $X_1,\ldots,X_m,Y,Z,T$ in $A(r_1,\ldots,r_m,f)$ and\\ $A(s_1,\ldots,s_m,g)$ respectively.
As $f(Z,T),g(Z,T)$ are non-trivial lines in $k[Z,T]$, by 
\thref{subdk} and \thref{r1}, we have $\dk(A(r_1,\ldots,r_m,f))=k[x_1,\ldots,x_m,z,t]$ and $\dk(A(s_1,\ldots,s_m,g))=k[x_1^{\prime},\ldots,x_m^{\prime},z^{\prime},t^{\prime}]$. Hence the result follows from \rm(ii) and \rm(iv) of \thref{isocl}.

The converse is obvious.
\end{proof}

\begin{cor}\thlabel{niso}
	Let $k$ be a field of positive characteristic. For each $n \geqslant 3$, there exists an infinite family of pairwise non-isomophic rings $C$ of dimension $n$, which are counter examples to the Zariski Cancellation Problem in positive characteristic, i.e., which satisfy that $C^{[1]}=k^{[n+1]}$ but $C \neq k^{[n]}$.
\end{cor}
\begin{proof}
	Consider the family of rings
	$$
	\Omega:=\{A(r_1,\ldots,r_m,f) \mid (r_1,\ldots,r_m) \in \mathbb{Z}^m_{>1}, \, f(Z,T)\, \text{is a non-trivial line in $k[Z,T]$}\}.
	$$  
	For every $C \in \Omega$, we have 	$C^{[1]} = k^{[m+3]}$
but
$C\neq k^{[m+2]}$ (\cite{adv}, Theorem 3.7). By \thref{isocl1}, 
there exist infinitely many rings $C\in \Omega$ which are pairwise non-isomorphic. Taking $n=m+2$, we get the result.
 \end{proof}

\section*{Acknowledgements}
The authors thank Professor Amartya K. Dutta for asking Question 2, carefully going through the earlier draft and giving valuable suggestions. The first author acknowledges the Council of Scientific and Industrial Research (CSIR) for the Shyama Prasad Mukherjee fellowship (SPM-07/0093(0295)/2019-EMR-I). The second author acknowledges Department of Science and Technology (DST), India for their INDO-RUSS project (DST/INT/RUS/RSF/P-48/2021).

\end{document}